\documentclass[a4paper,reqno]{amsart}

\usepackage{amscd,amsthm,amsmath,amssymb,mathtools}
\usepackage{upgreek,bm}
\usepackage{verbatim}
\usepackage{color}
\usepackage{url}
\usepackage{enumerate,enumitem}
\usepackage{graphicx}
\usepackage{algorithm}
\usepackage{algorithmic}
\usepackage{epstopdf,epsfig,subfigure}
\usepackage{curve2e}
\usepackage{mathrsfs}
\usepackage{array}
\usepackage{stackrel}

\usepackage[numbers,sort&compress]{natbib}

\usepackage{setspace}
\usepackage[
    bookmarks=true,         % show bookmarks bar?
    unicode=false,          % non-Latin characters in Acrobat’s bookmarks
    pdftoolbar=true,        % show Acrobat’s toolbar?
    pdfmenubar=true,        % show Acrobat’s menu?
    pdffitwindow=false,     % window fit to page when opened
    pdfstartview={FitH},    % fits the width of the page to the window
    pdftitle={Stabilization with delta distridutions},    % title
    pdfauthor={Author},     % author
    pdfsubject={Subject},   % subject of the document
    pdfcreator={Creator},   % creator of the document
    pdfproducer={Producer}, % producer of the document
    pdfkeywords={keyword1, key2, key3}, % list of keywords
    pdfnewwindow=true,      % links in new PDF window
    colorlinks=true,       % false: boxed links; true: colored links
    linkcolor=red,          % color of internal links (change box color with linkbordercolor)
    citecolor=blue,        % color of links to bibliography
    filecolor=magenta,      % color of file links
    urlcolor=cyan,           % color of external links
hypertexnames=false
]{hyperref}

\theoremstyle{plain}
 \newtheorem*{mainresult}{Main Result}

\newtheorem{theorem}{Theorem}[section]

\newtheorem{lemma}[theorem]{Lemma}
\newtheorem{corollary}[theorem]{Corollary}
\newtheorem{problem}[theorem]{Problem}
\newtheorem{assumption}[theorem]{Assumption}

\theoremstyle{definition}

\newtheorem{remark}[theorem]{Remark}

\numberwithin{equation}{section}

\usepackage{geometry}
\newcommand{\linspan}{\mathop{\rm span}\nolimits}

\newcommand{\rest}{\left.\kern-2\nulldelimiterspace\right|_}
\newcommand{\norm}[2]{\left|#1\right|_{#2}}
\newcommand{\dnorm}[2]{\left\|#1\right\|_{#2}}

\newcommand{\Id}{{\mathbf1}}

\newcommand{\ex}{\mathrm{e}}
\newcommand{\p}{\partial}
\newcommand{\ed}{\mathrm d}

\newcommand{\deltafun}{\bm\updelta}

\newcommand*{\Bigcdot}{\raisebox{-.25ex}{\scalebox{1.25}{$\cdot$}}}

\newcommand{\clA}{{\mathcal A}}

\newcommand{\clD}{{\mathcal D}}
\newcommand{\clE}{{\mathcal E}}

\newcommand{\clJ}{{\mathcal J}}

\newcommand{\clL}{{\mathcal L}}
\newcommand{\clM}{{\mathcal M}}

\newcommand{\clT}{{\mathcal T}}
\newcommand{\clU}{{\mathcal U}}

\newcommand{\clX}{{\mathcal X}}
\newcommand{\clY}{{\mathcal Y}}

\newcommand{\bbN}{{\mathbb N}}

\newcommand{\bbR}{{\mathbb R}}

\newcommand{\bfB}{{\mathbf B}}
\newcommand{\bfC}{{\mathbf C}}
\newcommand{\bfD}{{\mathbf D}}

\newcommand{\bfK}{{\mathbf K}}
\newcommand{\bfL}{{\mathbf L}}
\newcommand{\bfM}{{\mathbf M}}

\newcommand{\bfP}{{\mathbf P}}

\newcommand{\bfR}{{\mathbf R}}
\newcommand{\bfS}{{\mathbf S}}
\newcommand{\bfT}{{\mathbf T}}

\newcommand{\fkB}{{\mathfrak B}}
\newcommand{\fkC}{{\mathfrak C}}

\newcommand{\rmD}{{\mathrm D}}

\newcommand{\bfc}{{\mathbf c}}

\newcommand{\bff}{{\mathbf f}}

\newcommand{\bfh}{{\mathbf h}}

\newcommand{\bfn}{{\mathbf n}}

\newcommand{\bfu}{{\mathbf u}}

\newcommand{\bfx}{{\mathbf x}}

\newcommand{\rmd}{{\mathrm d}}
\newcommand{\rme}{{\mathrm e}}
\newcommand{\rmf}{{\mathrm f}}

\newcommand{\fkd}{{\mathfrak d}}

\newcommand{\ovlineC}[1]{\overline C_{\left[#1\right]}}

\definecolor{DarkBlue}{rgb}{0,0.08,0.45}
\definecolor{DarkRed}{rgb}{.65,0,0}
\definecolor{applegreen}{rgb}{0.55, 0.71, 0.0}

\newcounter{mymac@matlab}
\newcommand{\matlab}{MATLAB%
   \ifnum\value{mymac@matlab}<1%
   \textregistered%
   \setcounter{mymac@matlab}{1}%
   \fi%
  }

\newcommand{\black}{ \color{black} }

\begin{document}
\title{Stabilizability for nonautonomous linear parabolic equations with actuators as distributions}
\author{Karl Kunisch$^{\tt1,2}$,
\quad S\'ergio S.~Rodrigues$^{\tt1}$,
\quad Daniel Walter$^{\tt3}$}
\thanks{
\vspace{-1em}\newline\noindent
\newline\noindent
$^{\tt1}$ Johann Radon Institute for Comput. and Appl. Math.,
  \"OAW, Altenbergerstr. 69, 4040 Linz, Austria.
\newline\noindent
$^{\tt2}$ Institute for Mathematics and Scientific Computing, 621, Heinrichstrasse 36, 8010 Graz, Austria.
\newline\indent
$^{\tt3}$ Institute for  Mathematics, Humboldt University, Rudower Chaussee 25, 10117 Berlin,
Germany.
\newline\noindent
{\sc Emails}:
{\small\tt karl.kunisch@uni-graz.at},\quad
{\small\tt sergio.rodrigues@ricam.oeaw.ac.at,}\newline\hspace*{3.8em}{\small\tt daniel.walter@hu-berlin.de\,}.
}

\begin{abstract}
The  stabilizability  of a general class of abstract parabolic-like equations is
investigated, with a finite number of actuators. This class includes the case of actuators given as delta distributions located at given points in the spatial domain of concrete parabolic equations.
A stabilizing feedback control operator is constructed and given in explicit form. Then,  an associated optimal control is considered and the corresponding Riccati feedback is investigated. Results of  simulations are presented showing the stabilizing performance of both explicit and Riccati feedbacks.
\end{abstract}

\maketitle

\pagestyle{myheadings} \thispagestyle{plain} \markboth{\sc K. Kunisch,  S. S. Rodrigues, D. Walter}
{\sc Stabilization of  linear parabolic equations with delta distributions}

%%%%%%%%%%%%%%%%%%%%%%%%%%%
%%%%%%%%%%%%%%%%%%%%%%%%%%%
%%%%%%%%%%%%%%%%%%%%%%%%%%%
\section{Introduction}

We consider a general class of abstract parabolic-like equations, for time~$t>0$,
as
\begin{subequations}\label{sys-yfeed-intro}
\begin{equation}\label{sys-y-intro}
\dot y +A y+A_{\rm rc}y  =\textstyle\sum\limits_{j=1}^{M_\sigma}u_j\fkd_j,\qquad y(0)=y_0,
\end{equation}
where~$A$ is a diffusion-like operator, $A_{\rm rc}=A_{\rm rc}(t)$ is a time-dependent reaction--convection-like operator, and ~$u_j=u_j(t)$ are the coordinates of the control input~$u$
used to tune the actuators~$\fkd_j$, which are assumed to be elements in the continuous dual~$\rmD(A)'$ of the domain~$\rmD(A)$ of the (unbounded) operator~$A\colon H\to H$ in a pivot Hilbert space~$H$.
Under suitable further assumptions on these operators and actuators
we shall show that the explicitly given feedback input
\begin{equation}\label{feed-intro}
u_j(t)\coloneqq-\lambda\langle\fkd_{j},A^{-1}y\rangle_{\rmD(A)',\rmD(A)}
\end{equation}
\end{subequations}
is able to stabilize the system in the norm of a suitable space~$V'\supset H$. Without entering into more details at this point the main result reads as follows.
\begin{mainresult}%
 Let~$\mu>0$ be given. We can find a number~$M_\sigma$   of actuators, and a constant~$\lambda>0$ both large enough, so that
the solution of~\eqref{sys-yfeed-intro} satisfies
\begin{equation}\notag
\norm{y(t)}{V'}\le\ex^{-\mu (t-s)}\norm{y(s)}{V'},\quad\mbox{for all   }t\ge s\ge0\mbox{ and all }y_0\in V'.
\end{equation}
In particular, $t\mapsto \norm{y(t)}{V'}$ is strictly decreasing at time~$t=s$, if~$\norm{y(s)}{V'}\ne0$. \end{mainresult}
The details will be given in a more precise formulation in Theorem~\ref{T:main}.

%%%%%%%%%%%%%%%%%%%%
%%%%%%%%%%%%%%%%%%%%
\subsection{Example of application}
We shall show that we can apply  the abstract result to concrete parabolic equations as
\begin{subequations}\label{sys-y-parab-intro-expl}
\begin{align}
 &\tfrac{\p}{\p t} y -\nu\Delta y+ay +b\cdot\nabla y
 = -\lambda\textstyle\sum\limits_{j=1}^{M_\sigma}\langle\deltafun_{x^j},A^{-1}y(\Bigcdot,t)\rangle_{\rmD(A)',\rmD(A)}\deltafun_{x^j},\\
  &\fkB y\rest{\p\Omega}=0,\quad
y(0)=y_0,\quad\mbox{with}\quad A\coloneqq -\nu\Delta+\Id,
       \end{align}
 \end{subequations}
evolving in rectangular spatial domains~$\Omega\subset\bbR^d$ with~$d\in\{1,2,3\}$, under Dirichlet or Neumann boundary conditions, where~$\Id$ stands for the identity operator.
Denoting by~$(x,t)$ a generic point in the cylinder~$\Omega\times(0,+\infty)$, the reaction~$a=a(x,t)\in\bbR$ and the convection~$b=b(x,t)\in\bbR^d$ coefficients are known  functions. In general (e.g., due to the reaction term), the norm of the state corresponding to the free dynamics (i.e., to $u=0$) can diverge to~$+\infty$ as time increases.

We shall show that system~\eqref{sys-y-parab-intro-expl} is
stable for appropriately chosen locations~$x^j\in\Omega$ of the actuators, given by delta distributions~$\deltafun_{x^j}$ supported at the points~$x^j$.
In this concrete example, we focus on the cases~$d\in\{1,2,3\}$, which are the most relevant for several real world applications. The stability result, as stated in Main Result, will hold in the norm of the continuous dual~$V'$ of the Sobolev space~$V= W^{1,2}(\Omega)\subset L^2(\Omega)$. The details shall be given in Section~\ref{S:parab}.

\begin{enumerate}[leftmargin=1em,itemindent=1em,label={\bf\arabic*)}]
\item[] \hspace{-1em}We end this section with some remarks.
\item[--] For the moment, we may think of~$M_\sigma=M$. Later on, $M_\sigma=\sigma(M)$ will be a
 more general strictly increasing sequence of positive integers, which will be convenient, in concrete applications, to achieve suitable auxiliary properties for(a sequence of) families of actuators.
\item[--] For simplicity, we consider the case of rectangular domains. The extension to more general convex polygonal domains shall be addressed in Remark~\ref{R:ext-polygonal}.
\item[--] We shall check the satisfiability of the abstract assumptions, provided later on, for concrete actuators as above given as delta  distributions supported at single points. We can consider more general concrete actuators in~$\rmD(A)'$ (e.g., distributions supported in a curve or surface); to conclude the desired stability of the system, we will ``only'' need to check the satisfiability of the same assumptions by the new concrete actuators.
\end{enumerate}

%%%%%%%%%%%%%%%%%
%%%%%%%%%%%%%%%%%
\subsection{Related work} Here we focus on literature involving actuators which are not elements of the pivot space. Stabilization problems with finitely many actuators in the pivot space  are considered in~\cite{Azmi22,KunRodWal21,BarbuTri04,Barbu03,AzouaniTiti14}, for example.

Apparently the problem of   stabilization of nonautonomous parabolic equations with actuators given by a finite number of fixed delta distributions~$\deltafun_{x^j}$, $1\le j\le M_\sigma$ has not been investigated in the literature before. However, there are results on the
control  of parabolic equations in the one-dimensional case ~$\Omega=(0,l)\subset\bbR$, $l>0$,  with a mobile $\delta$ distribution~$\deltafun_{\overline x(t)}$ located at the time dependent point~$\overline x(t)\in\Omega$.  For example, see~\cite{Khapalov01}. Note that, in the 1D case we have that~$\deltafun_{\overline x(t)}\in W^{-1,2}(\Omega)$, at each time instant~$t$, where~$W^{-1,2}(\Omega)$ is  the continuous dual space of the Sobolev subspace~$W_0^{1,2}(\Omega)$, which means that we  have enough spatial regularity in order to have existence of standard weak solutions~$y$ for~\eqref{sys-y-parab-intro-expl} (cf.~\cite[Appendix]{Khapalov01}). In this manuscript we consider the case of static actuators and, further, also the cases~$\Omega\subset\bbR^d$, $d\in\{2,3\}$, where~$\deltafun_{x^j}\notin W^{-1,2}(\Omega)$; this means that, for actuators as~$\deltafun_{x^j}$ we will have to work with less regular solutions, with respect to the spatial variable.
 The case of delta distributions as actuators  has also been considered in~\cite{CastroZuazua05} where approximate controllability results have been derived. Note that approximate controllability to zero at a given finite time-horizon~$T>0$ does not imply stabilizability.
In the context of the wave equation, for  domains~$\Omega\subset\bbR^2$,  moving actuators ~$\deltafun_{\bfx(t)}$ acting as a weak damping,
 with support is not a single point, but a closed simple curve~$\bfx(t)\subset\Omega$,  were considered in~\cite{JaffardTucsnakZuazua98}. These distributions are more regular  than point supported ones.  Indeed,~$\deltafun_{\bfx(t)}\in W^{-1,2}(\Omega)$,
$\langle\deltafun_{\bfx(t)},z\rangle\coloneqq\int_{\bfx(t)}z\,\rmd\bfx(t)$ for~$z\in W^{1,2}(\Omega)$.

For~$d\in\{1,2,3\}$, the point delta distributions~$\deltafun_{x^j}\in \rmD(A)'\subset W^{-2,2}(\Omega)$ belong to the continuous dual~$\rmD(A)'$ of the domain~$\rmD(A)$ of the operator~$A$ as in~\eqref {sys-y-parab-intro-expl} (for both Dirichlet and Neumann boundary conditions). The dynamics in~\eqref{sys-y-intro} can be written as~$\dot y=\clA y+Bu$, with~$\clA =-A-A_{\rm rc}$ and a continuous linear control operator~$B\colon\bbR^{M_\sigma}\to\rmD(A)'$.  Though, we are considering the (more general) nonautonomous case~$\clA=\clA(t)$, we would like to refer to~\cite{BadraTakahashi14} where a stabilizability criterion is presented for the autonomous case with~$B\colon\bbR^{M_\sigma}\to\rmD(\clA^*)'$, where~$\clA^*$ stands for the adjoint of~$\clA$; see also~\cite{Raymond19}.

%%%%%%%%%%%%%%%%%%%%
%%%%%%%%%%%%%%%%%%%%
\subsection{Contents and notation}
The remaining text of the manuscript is organized as follows. In  Section~\ref{S:assumptions} we introduce the general assumptions defining a class of abstract parabolic-like systems under which we prove the stabilizability result. This proof is presented in Section~\ref{S:stab}. Then, we consider a related infinite time-horizon optimal control problem in Section~\ref{S:optimal_loc} and investigate associated differential Riccati equations. The applicability of the abstract result to concrete scalar parabolic equations is shown in Section~\ref{S:parab}. Section~\ref{S:discRicc} includes some comments on the discretization used to compute the solutions of the Riccati equations. The stabilizing performances of the explicit and  the Riccati  feedbacks are shown and discussed in Section~\ref{S:numRes} based on results of numerical simulations. We end the manuscript with concluding remarks in Section~\ref{S:finalremks}.

\bigskip
Concerning notation, we
write~$\bbR$ and~$\bbN$ for the sets of real numbers and nonnegative
integers, respectively, and we set $\bbR_+\coloneqq(0,\infty)$,
and $\bbN_+\coloneqq\mathbb N\setminus\{0\}$.

Given Banach spaces~$X$ and~$Y$, we write $X\xhookrightarrow{} Y$ if
$X\subseteq Y$ is a continuous inclusion. We write
$X\xhookrightarrow{\rm d} Y$, and $X\xhookrightarrow{\rm c} Y$,
if the inclusion is also dense, respectively compact.

We define the Bochner subspace~$W(I,X,Y)\coloneqq\{f\in L^2(I,X)\mid \dot f\in L^2(I,Y)\}$, for a given nonempty open interval~$I\subset\bbR_+$.

The space of continuous linear mappings from~$X$ into~$Y$ is denoted by~$\clL(X,Y)$. In case~$X=Y$ we
write~$\clL(X)\coloneqq\clL(X,X)$.
The continuous dual of~$X$ is denoted~$X'\coloneqq\clL(X,\bbR)$.
The adjoint of an operator $L\in\clL(X,Y)$ will be denoted $L^*\in\clL(Y',X')$.

By $C,\,C_i$, $i=0,\,1,\,\dots$, we shall denote several positive nonessential constants (i.e., which may differ at different places in the manuscript). To underline the dependence of a constant (e.g.,  an upper bound) on some parameters, we use
$\overline C_{\left[a_1,\dots,a_n\right]}$ denoting a nonnegative function that
increases in each of its nonnegative arguments~$a_i$, $1\le i\le n$.

%%%%%%%%%%%%%%%%%%%%%%%%%
%%%%%%%%%%%%%%%%%%%%%%%%%
%%%%%%%%%%%%%%%%%%%%%%%%%
\section{Assumptions}\label{S:assumptions}
We introduce the general class of evolutionary parabolic-like equations as
\begin{align}\label{sys-y}
\dot y +A y+A_{\rm rc}y  =-\lambda\textstyle\sum\limits_{j=1}^{M_\sigma}\langle\fkd_{j},A^{-1}y\rangle\fkd_j,\qquad y(0)=y_0,
\end{align}
where the
operators~$A$, $A_{\rm rc}=A_{\rm rc}(t)$ appearing in the system dynamics, and
the actuators~$\fkd_j$ are asked to satisfy appropriate assumptions.

First of all we are given two Hilbert spaces~$V\subset H=H'$. The first two assumptions require~$A$ to be a  (time-independent) diffusion-like operator.
\begin{assumption}\label{A:A0sp}
 $A\in\clL(V,V')$ is symmetric, and such that $(y,z)\mapsto\langle Ay,z\rangle_{V',V}$ is a
 complete scalar product on~$V.$
\end{assumption}

Hereafter, we suppose that~$V$ is endowed with the scalar product~$(y,z)_V\coloneqq\langle Ay,z\rangle_{V',V}$,
which again makes~$V$ a Hilbert space.
Necessarily, $A\colon V\to V'$ is an isometry.
\begin{assumption}\label{A:A0cdc}
The inclusion $V\subseteq H$ is dense, continuous, and compact.
\end{assumption}

Necessarily, the operator $A$ is densely defined in~$H$, with domain $\rmD(A)$ satisfying
\begin{equation}\notag
\rmD(A)\xhookrightarrow{\rm d,\,c} V\xhookrightarrow{\rm d,\,c} H\xhookrightarrow{\rm d,\,c} V'\xhookrightarrow{\rm d,\,c}\rmD(A)'.
\end{equation}
Further,~$A$ has a compact inverse~$A^{-1}\colon H\to H$, and we can find a nondecreasing
sequence of  eigenvalues $(\alpha_n)_{n\in\bbN_+}$ (repeated accordingly to their multiplicity) and a corresponding basis of
eigenfunctions $(e_n)_{n\in\bbN_+}$:
\begin{equation}\label{eigfeigv}
0<\alpha_1\le\alpha_2\le\dots\le\alpha_n\to+\infty, \quad Ae_n=\alpha_n e_n.
\end{equation}

Observe that we have the relations
 \begin{align}\notag
&\langle y,z\rangle_{X',X}=(y,z)_H,&&\quad\mbox{for all}\quad(y,z)\in H\times X,\qquad X\in\{V,\rmD(A)\};\\
  &\langle y,z\rangle_{\rmD(A)',\rmD(A)}=\langle y,z\rangle_{V',V},&&\quad\mbox{for all}\quad(y,z)\in V'\times \rmD(A).\notag
\end{align}
Hence, without ambiguity, we may omit the subscripts and will often write, for simplicity,
\begin{align}\notag
  &\langle y,z\rangle\coloneqq\langle y,z\rangle_{X',X},\quad\mbox{for } (y,z)\in X'\times X,\qquad X\in\{V,\rmD(A)\}.
\end{align}

The next assumption requires~$A_{\rm rc}=A_{\rm rc}(t)$ to be a time-dependent reaction--convection-like operator.
\begin{assumption}\label{A:A1}
For almost every~$t>0$ we have~$A_{\rm rc}(t)\in\clL(H, V')$,
and we have a uniform bound as $\norm{A_{\rm rc}}{L^\infty(\bbR_0,\clL(H,V'))}\eqqcolon C_{\rm rc}<+\infty.$
\end{assumption}

Finally, we make an assumption on the sequence of the families of actuators.

\begin{assumption}\label{A:poincare}
The sequence~$(U_M)_{M\in\bbN_+}$ of families~${U_{M}}\coloneqq\{\fkd^{M,j}\mid 1\le j\le M_\sigma\}$ of actuators satisfy the following.
 \begin{enumerate}[leftmargin=1em,itemindent=1em,label={\bf\arabic*)}]
\item
$M_\sigma\coloneqq\sigma(M)$, where~$\sigma\colon\bbN_+\to\bbN_+$
is a strictly increasing function;
\item
${U_{M}}\subset \rmD(A)'$ and~$\clU_{M}\coloneqq\linspan U_M$ has dimension~$\dim \clU_M=M_\sigma$,
\item\label{assu.actaux-seq} there exists a sequence~$(\widetilde U_M)_{M\in\bbN_+}$ of families~$\widetilde U_{M}\coloneqq\{\Psi^{M,j}\mid 1\le j\le M_\sigma\}$ of auxiliary functions, satisfying
 \begin{enumerate}[leftmargin=1em,itemindent=1em,label={\bf\alph*)}]
\renewcommand{\theenumii}{\ref{assu.actaux-seq}{\bf\alph{enumii})}}
%for appearing in future ref calling
%  \renewcommand{\labelenumi}{}%for appearing inside enumerate env, trick to avoid indentation
\item
${\widetilde U_{M}}\subset \rmD(A)$ and~$\widetilde \clU_{M}\coloneqq\linspan\widetilde  U_M$ has dimension~$\dim \widetilde \clU_M=M_\sigma$,
\item\label{actaux_ij}
$\langle\fkd^{M,j},\Psi^{M,i}\rangle=\begin{cases}1&\mbox{ if }j=i,\\ 0&\mbox{ if }j\ne i,\end{cases}$
\item\label{actaux_DM}
with~$\rmD_M\coloneqq\{z\in\rmD(A)\mid \langle\fkd^{M,j},z\rangle=0\mbox{ for all }1\le j\le M_\sigma\}$, the   constant
\begin{align}\label{Poinc_const}
\xi_{M_+}\coloneqq\inf_{\varTheta\in\rmD_M\setminus\{0\}}
\tfrac{\norm{\varTheta}{\rmD(A)}^2}{\norm{\varTheta}{V}^2},
\end{align}
satisfies
$
\lim\limits_{M\to+\infty}\xi_{M_+}=+\infty.
$
\end{enumerate}
\end{enumerate}
\end{assumption}

\begin{remark}
The satisfiability of Assumptions~\ref{A:A0sp}--\ref{A:poincare} for scalar parabolic
systems as~\eqref{sys-y-parab-intro-expl} is addressed in Section~\ref{S:parab}.
\end{remark}

%%%%%%%%%%%%%%%%%%%%%%%
%%%%%%%%%%%%%%%%%%%%%%%
%%%%%%%%%%%%%%%%%%%%%%%
\section{The main stabilizability result}\label{S:stab}
The following Theorem~\ref{T:main} is a precise statement of Main Result stated in the Introduction. This section is mainly dedicated to its proof.

We start by introducing the mapping
\begin{equation}\label{Psi-map}
\Psi^\diamond\colon\bbR^{M_\sigma}\to\widetilde\clU_M,\qquad \Psi^\diamond v\coloneqq{\textstyle\sum\limits_{j=1}^{M_\sigma}}v_j\Psi^{M,j},
\end{equation}
where~$\widetilde\clU_M=\linspan\widetilde U_M$ is the linear span of the auxiliary functions as in Assumption~\ref{A:poincare}.
Since the family~$\widetilde U_M$ is linearly independent, we have that~$\norm{\Bigcdot}{\bbR^{M_\sigma}}$, $\norm{\Psi^\diamond \Bigcdot}{V}$, and~$\norm{\Psi^\diamond \Bigcdot}{\rmD(A)}$, are three norms in~$\bbR^{M_\sigma}$. Hence, there exists a constant~$\varpi_M>0$ such that
\begin{equation}\label{varpiM}
\norm{\Psi v}{\rmD(A)}^2\le \varpi_M\norm{\Psi^\diamond v}{V}^2\quad\mbox{and}\quad\norm{\Psi^\diamond v}{V}^2\le \varpi_M\norm{v}{\bbR^{M_\sigma}}^2,\qquad\mbox{for every}\quad v\in\bbR^{M_\sigma}.
\end{equation}

\begin{theorem}\label{T:main}
 Let Assumptions~\ref{A:A0sp}--~\ref{A:poincare} hold true. Then, for every~$\mu>0$ we can find if~$M\in\bbN_+$ and~$\lambda>0$ large enough,
so that the solution of system~\eqref{sys-y}
 satisfies
\begin{align}
& y\in W((0,+\infty),L^2,\rmD(A)')\quad\mbox{and}\notag\\
&\norm{y(t)}{V'}\le\ex^{-\mu (t-s)}\norm{y(s)}{V'},\quad\mbox{for all } t\ge s\ge0\mbox{ and all } y_0\in V'. \notag
\end{align}
Furthermore,~$M$ and~$\lambda$ can be chosen as~$M=\ovlineC{\mu,C_{\rm rc}}$
and~$\overline\lambda=\ovlineC{\mu,C_{\rm rc}}$. In addition, we have the bound
$\norm{y}{L^2((s,+\infty),L^2)}^2
\le (2+{\mu}^{-1}{C_{\rm rc}^2})\norm{y(s)}{V'}^2$, for all~$s\ge0$.
\end{theorem}

\begin{remark}
In particular, the estimate in Theorem~\ref{T:main} implies that $t\mapsto \norm{y(t)}{V'}$ is strictly decreasing at time~$t=s$, if~$\norm{y(s)}{V'}\ne0$; see~\cite[Lem.~3.3]{Rod21-aut}.
\end{remark}

For the proof of Theorem~\ref{T:main} we will need an auxiliary result.
For this purpose, for a given~$M\in\bbN_+$ and~$z\in\rmD(A)$,  we denote the vector
\begin{equation}\notag
\fkd(z)\coloneqq \left( \langle\fkd^{M,1},z\rangle,\langle\fkd^{M,2},z\rangle,\;\dots\,,\langle\fkd^{M,M_\sigma},z\rangle\right)\in\bbR^{M_\sigma}
\end{equation}
where the~$\fkd^{M,j}\in\rmD(A)'$ are the actuators.

\begin{lemma}\label{L:OblDec}
Every~$z\in\clD(A)$, can be written in a unique way as
\begin{align}\notag
z=\vartheta+\varTheta,\quad\mbox{with}\quad
\vartheta\coloneqq \Psi^\diamond\fkd(z)\in\widetilde\clU_{M}\quad\mbox{and}\quad
\varTheta\coloneqq z-\vartheta\in\rmD_M,
\end{align}
where~$\widetilde\clU_{M}$ and~$\rmD_M$ are as in Assumption~\ref{A:poincare}.
\end{lemma}
\begin{proof}
We have that~$\vartheta\in\widetilde\clU_{M}$ due to the definition of~$\Psi^\diamond$ in~\eqref{Psi-map}. Next, by the relations
\begin{align}\notag
\langle\fkd^{M,i},\varTheta\rangle=\Bigl\langle\fkd^{M,i},\,z-{\textstyle\sum\limits_{j=1}^{M_\sigma}}\langle\fkd^{M,j},z\rangle\Psi^{M,j}\Bigr\rangle,
\end{align}
and by~\ref{actaux_ij} in Assumption~\ref{A:poincare}, we find~$\langle\fkd^{M,i},\varTheta\rangle=0$ for all~$\fkd^{M,i}$, hence
$\varTheta\in\rmD_M$. It remains to show that~$\widetilde\clU_{M}\bigcap D_M=\{0\}$, for this purpose
let us be given~$z\in\widetilde\clU_{M}\bigcap D_M$, then it follows that there exists~$\overline z\in\bbR^{M_\sigma}$ such that
\begin{align}\notag
 z={\textstyle\sum\limits_{j=1}^{M_\sigma}}\overline z_j\Psi^{M,j}\quad\mbox{and}\quad \Bigl\langle\fkd^{M,i},\,z\Bigr\rangle=0\quad\mbox{for all}\quad 1\le i\le M_\sigma.
\end{align}
Again by~\ref{actaux_ij} in Assumption~\ref{A:poincare}, we find $\overline z_i=\Bigl\langle\fkd^{M,i},\,z\Bigr\rangle$, which leads us to~$z=0$.
\end{proof}

\begin{lemma}\label{L:MlamPoinc}
Let Assumptions~\ref{A:A0sp}, ~\ref{A:A0cdc} and~\ref{A:poincare} hold true. Then, for every~$\zeta>0$ we can find~$M$ and~$\lambda$ large enough such that
\begin{equation}
\norm{z}{\rmD(A)}^2+2\lambda \norm{\fkd(z)}{\bbR^{M_\sigma}}^2
\ge\zeta \norm{z}{V}^2, \quad\mbox{for all}\quad
z\in \rmD(A).
\end{equation}
Furthermore~$M=\ovlineC{\zeta}$  and~$\lambda=\ovlineC{\zeta,\varpi_M}$, with~$\varpi_M$ as in~\eqref{varpiM}.
 \end{lemma}
\begin{proof}
Recall the sets~$U_M$ and~$\widetilde U_M$ of actuators~$\fkd^{M,j}$ and auxiliary functions~$\Psi^{M,j}$ and their linear spans as~$\clU_M$ and~$\widetilde \clU_M$, as well as the space~$\rmD_M$ in Assumption~\ref{A:poincare}.

For each~$z\in\clD(A)$, due to Lemma~\ref{L:OblDec}, we  can write
\begin{align}\notag
z=\vartheta+\varTheta,\quad\mbox{with}\quad
\vartheta\coloneqq \Psi^\diamond\fkd(z)\in\widetilde\clU_{M}\quad\mbox{and}\quad
\varTheta\coloneqq z-\vartheta\in\rmD_M.
\end{align}

Now, by direct computations, we find
\begin{align}
&\norm{z}{\rmD(A)}^2+2\lambda\norm{\fkd(z)}{\bbR^{M_\sigma}}^2 =\norm{\varTheta+\vartheta}{\rmD(A)}^2+2\lambda \norm{\fkd(z)}{\bbR^{M_\sigma}}^2\notag\\
&\hspace{3em}=\norm{\varTheta}{\rmD(A)}^2+2(\varTheta,\vartheta)_{\rmD(A)}+\norm{\vartheta}{\rmD(A)}^2+2\lambda \norm{\fkd(z)}{\bbR^{M_\sigma}}^2\notag\\
&\hspace{3em}\ge\tfrac12\norm{\varTheta}{\rmD(A)}^2-\norm{\vartheta}{\rmD(A)}^2+2\lambda \norm{\fkd(z)}{\bbR^{M_\sigma}}^2\notag
\end{align}
and, with~$\varpi_M$ as in~\eqref{varpiM} and~$\xi_{M_+}$ as in~\eqref{Poinc_const}, we arrive at
\begin{align}
\norm{z}{\rmD(A)}^2+2\lambda\norm{\fkd(z)}{\bbR^{M_\sigma}}^2
&\ge\tfrac12\xi_{M_+}\norm{\varTheta}{V}^2+(2\lambda\varpi_M^{-1} -\varpi_M)\norm{\vartheta}{V}^2.\notag
\end{align}
Hence, for given~$\zeta>0$, by choosing~$M$ and~$\lambda$ so that
\begin{equation}\notag
\xi_{M_+}\ge 4\zeta\quad\mbox{and}\quad\lambda\ge \zeta\varpi_M^{2}
\end{equation}
we arrive at
\begin{align}
\norm{y}{\rmD(A)}^2+2\lambda\norm{\fkd(z)}{\bbR^{M_\sigma}}^2
&\ge2\zeta\left(\norm{\varTheta}{V}^2+ \norm{\vartheta}{V}^2\right)
\ge \zeta\norm{\varTheta+\vartheta}{V}^2=\zeta\norm{z}{V}^2,\notag
\end{align}
which ends the proof.
\end{proof}

\begin{proof}[Proof of Theorem~\ref{T:main}]
Let~$X_{\rm rc}\coloneqq A^{-1}A_{\rm rc}A$ and~$z=A^{-1}y$ where~$y$ satisfies~\eqref{sys-y}.
Thus, $z$ satisfies
\begin{align}\label{sys-z}
&\dot z +A z+X_{\rm rc}z  =-\lambda\textstyle\sum\limits_{j=1}^{M_\sigma}\langle\fkd^{M,j},z\rangle A^{-1}\fkd^{M,j},\qquad
z(0)= A^{-1}y_0\in V.
       \end{align}

Multiplying the dynamics   in~\eqref{sys-z} by~$2Az$,
 we find\black
\begin{align}
\tfrac{\ed}{\ed t}\norm{z}{V}^2 &\le-2\norm{z}{\rmD(A)}^2-2(X_{\rm rc}z,Az)_{L^2}
-2\lambda \norm{\fkd(z)}{\bbR^{M_\sigma}}^2\notag
\end{align}
and, using Assumption~\ref{A:A1} and the Young inequality, we obtain,
\begin{align}
\tfrac{\ed}{\ed t}\norm{z}{V}^2 &\le-2\norm{z}{\rmD(A)}^2-2\lambda \norm{\fkd(z)}{\bbR^{M_\sigma}}^2+2C_{\rm rc}\norm{Az}{L^2}\norm{z}{V}\label{est-dtz0}\\
&\le-\norm{z}{\rmD(A)}^2-2\lambda \norm{\fkd(z)}{\bbR^{M_\sigma}}^2+C_{\rm rc}^2\norm{z}{V}^2,\label{est-dtz}
\end{align}
where we have used~$(X_{\rm rc}z,Az)_{L^2}=\langle A_{\rm rc}Az,A^{-1}Az\rangle_{V',V}=\langle A_{\rm rc}Az,z\rangle_{V',V}$.
Now, for a given~$\mu>0$, by Lemma~\ref{L:MlamPoinc} we can choose~$M=\ovlineC{\mu,C_{\rm rc}}$ and~$\lambda=\ovlineC{\mu,C_{\rm rc},\varpi_M}$,  such that
\begin{align}\label{DAdelV}
\norm{z}{\rmD(A)}^2+2\lambda \norm{\fkd(z)}{\bbR^{M_\sigma}}^2\ge (2\mu+C_{\rm rc}^2)\norm{z}{V}^2
\end{align}
which gives us
\begin{align}
\tfrac{\ed}{\ed t}\norm{z}{V}^2
&\le-2\mu\norm{z}{V}^2\notag
\end{align}
and consequently
\begin{equation}\label{nz.V}
\norm{z(t)}{V}^2\le\ex^{-2\mu(t-s)}\norm{z(s)}{V}^2,\quad\mbox{for all}\quad t\ge s\ge0,
\end{equation}
which, due to~$y=Az$, is equivalent to
\begin{equation}\label{ny.V'}
\norm{y(t)}{V'}^2\le\ex^{-2\mu(t-s)}\norm{y(s)}{V'}^2,\quad\mbox{for all}\quad t\ge s\ge0.
\end{equation}

Finally, by~\eqref{est-dtz0}, Young inequality, and~\eqref{DAdelV} ,
\begin{align}
\tfrac{\ed}{\ed t}\norm{z}{V}^2 &\le-\tfrac32\norm{z}{\rmD(A)}^2-2\lambda \norm{\fkd(z)}{\bbR^{M_\sigma}}^2+2C_{\rm rc}^2\norm{z}{V}^2\le-\tfrac12\norm{z}{\rmD(A)}^2+C_{\rm rc}^2\norm{z}{V}^2\notag
\end{align}
which, after time integration, gives us, for each~$t\ge s\ge0$,
\begin{equation}\notag
\norm{z(t)}{V}^2-\norm{z(s)}{V}^2+\tfrac12\norm{z}{L^2((s,t),\rmD(A))}^2\le C_{\rm rc}^2\norm{z}{L^2((s,t),V)}^2,
\end{equation}
and taking the limit as~$t\to+\infty$, and using~\eqref{nz.V},
\begin{equation}\label{bddL2DA}
\norm{z}{L^2((s,+\infty),\rmD(A))}^2\le 2\norm{z(s)}{V}^2+2C_{\rm rc}^2\norm{z}{L^2((s,+\infty),V)}^2
\le (2+2C_{\rm rc}^2(2\mu)^{-1})\norm{z(s)}{V}^2.
\end{equation}
Finally, with~$\clY\coloneqq L^2((s,+\infty),L^2)$, by combining~\eqref{bddL2DA} with~\eqref{sys-z}, we arrive at
\begin{align}
\norm{\dot z}{\clY}&\le \norm{Az}{\clY}+\norm{X_{\rm rc}z}{\clY}+\norm{\lambda\textstyle\sum_{j=1}^{M_\sigma}\langle\fkd^{M,j},z\rangle A^{-1}\fkd^{M,j}}{\clY}\notag\\
&\le\left(1+\norm{\Id}{\clL(V',\rmD(A)')}C_{\rm rc}+\lambda M_\sigma\dnorm{\fkd}{}\right)\norm{z}{L^2((s,+\infty),\rmD(A))},\notag
\end{align}
with~$\dnorm{\fkd}{}\coloneqq\max\limits_{1\le j\le M_\sigma}\dnorm{\fkd^{M,j}}{\rmD(A)'}$.
Therefore, ~$z\in W((s,+\infty),\rmD(A),L^2)$, which gives us that~$y=A^{-1}z\in W((s,+\infty),L^2,\rmD(A)')$ and, by~\eqref{bddL2DA}, also that~$\norm{y}{L^2((s,\infty),L^2)}\le (2+C_{\rm rc}^2\mu^{-1})\norm{y(s)}{V'}^2$  for all~$s\ge0$.
\end{proof}

\begin{remark}[On the existence and uniqueness of solutions]
Within the proof of Theorem~\ref{T:main}, we have shown the stability of
system~\eqref{sys-z} for~$M$ and~$\lambda$ large enough, where we have implicitly
assumed that strong solutions do exist. The existence and uniqueness of these solutions can be proven
by classical arguments based on Galerkin approximations. Since the procedure is standard, we  simply recall the main steps.
The estimates in the proof of Theorem~\ref{T:main} also hold for
Galerkin approximations of system~\eqref{sys-z} given by
\begin{align}\label{sys-y-intro-K-GalN}
 &\dot z^N +Az^N+ P_{\clE_N^{\rm f}}A_{\rm rc}(t)z^N =-\lambda P_{\clE_N^{\rm f}}\textstyle\sum\limits_{j=1}^{M_\sigma}\langle\fkd^{M,j},z\rangle A^{-1}\fkd^{M,j},\\
 & z^N(0)=  P_{\clE_N^{\rm f}}(A^{-1}y_0),
\end{align}
where~$P_{\clE_N^{\rm f}}\in\clL(H)$ stands  for the orthogonal projection in~$H$ onto the subspace~$\clE_N^{\rm f}=\linspan\{e_i\mid 1\le i\le N\}$ spanned by the first eigenfuntions of~$A$.

Therefore, for arbitrary given~$T>0$, by the analogue to~\eqref{nz.V} we will have that for~$M$ and~$\lambda$ large enough it holds that
\begin{equation}\label{dtny.V-Gal}
\norm{z^N}{L^\infty((0,T),V)}^2\le \norm{z^N(0)}{V}^2\le  \norm{z(0)}{V}^2.
\end{equation}
Then, by the analogue of~\eqref{est-dtz}, we can find that
\begin{equation}\label{dtny.DAV-Gal}
\norm{z^N}{L^2((0,T),\rmD(A))}^2\le C_2\norm{z^N(0)}{V}^2,
\end{equation}
with~$C_2$ independent of~$N$. Next, we can use the dynamics~\eqref{sys-y-intro-K-GalN} to obtain
\begin{equation}\label{dtny.H-Gal}
\norm{\dot z^N}{L^2((0,T),L^2)}^2\le C_3\norm{z^N(0)}{V}^2
\end{equation}
with~$C_2$ independent of~$N$.
These uniform (in~$N$) estimates above, allow us to show that a weak limit of a suitable subsequence of such Galerkin approximations is a
strong solution for system~\eqref{sys-z} defined on the time interval~$I_T\coloneqq(0,T)$.
The uniqueness of the solution can be concluded by~\eqref{nz.V},
which clearly also holds for the difference of two solutions, due to the linearity of the dynamics.
Finally,  since~$T>0$ is arbitrary, we can conclude the existence and uniqueness of a strong solution defined for all time~$t>0$.
\end{remark}

%%%%%%%%%%%%%%%%%%%%%%%%%%%%
%%%%%%%%%%%%%%%%%%%%%%%%%%%%
%%%%%%%%%%%%%%%%%%%%%%%%%%%%
\section{Optimal control and Riccati equations}\label{S:optimal_loc}
In some applications we may be interested in stabilizing controls minimizing a given cost functional, in this section we discuss such an optimal control problem.

We shall address the Riccati equations associated to the optimal controls as well. Though the procedure is classical, this is not a standard problem in our setting due to the low spatial regularity of the state and control.

We fix a general set of~$M_0$ actuators
\begin{equation}\notag
U_{M_0}\coloneqq\{\fkd_{j}\mid 1\le j\le M_0\}\subset\rmD(A)',\qquad\clU_{M_0}\coloneqq\linspan U_{M_0},\qquad\dim\clU_{M_0}=M_0,
\end{equation}
and consider the open-loop version of system~\eqref{sys-y}, for time~$t\ge s$, as
\begin{align}\label{sys-y-opt-s-dyn}
 &\dot y +A y+A_{\rm rc}  y
 =Bu,\qquad y(s)= y_s,\qquad\mbox{with}\qquad Bu\coloneqq\textstyle\sum\limits_{j=1}^{M_0} u_j\fkd_{j}.
\end{align}
Note that $B\in \mathcal{L}(\mathbb{R}^{M_0},\rmD(A)')$. We assume that such actuators allow us to stabilize the system with a control input~$u=Ky$ given in feedback form, for example,  we can take actuators as given in Theorem~\ref{T:main}, with~$ M_0=M_\sigma$ large enough and~$K=(K_1,K_2,\dots,K_{M_\sigma})$, with~$K_jy=-\lambda\langle\fkd_{j},A^{-1}y\rangle$ as in~\eqref{sys-y}. In particular,  for any fixed~$\beta>0$, the energy functional
\begin{equation}\label{clJ-ori}
\widehat\clJ_s(y_s;y,u)\coloneqq\tfrac12\norm{y}{L^2((s,+\infty),H)}^2+\tfrac12\beta\norm{u}{L^2((s,+\infty),\bbR^{M_0})}^2
\end{equation}
will be bounded for the stabilizing feedback control input~$u=Ky$ and corresponding state~$y$. Thus, it makes sense to look for controls minimizing this  functional. Here, we shall relax this goal and shall rather be  looking for controls minimizing the ``relaxed'' functional
\begin{equation}\label{clJ}
\clJ_s(y_s;y,u)\coloneqq\tfrac12\norm{P_{\clE_{M_1}^\rmf}y}{L^2((s,+\infty),H)}^2+\tfrac12\beta\norm{u}{L^2((s,+\infty),\bbR^{M_0})}^2
\end{equation}
 where~$P_{\clE_{M_1}^\rmf}$ is the orthogonal projection in~$H$ onto the space
\begin{equation}\notag
\clE_{M_1}^\rmf=\linspan\{e_i\mid 1\le i\le M_1\},
\end{equation}
spanned by the first eigenfunctions~$e_i$ of the operator~$A$; see Section~\ref{S:assumptions}. We shall show in Theorem~\ref{T:bound-relM1-ori} that, for large enough~$M_1$, the boundedness of the relaxed  cost~\eqref{clJ} implies the boundedness of the original cost~\eqref{clJ-ori}. This relaxation can be important for numerical  computations,
We  shall revisit this point in Section~\ref{sS:rank-fact}.

\begin{theorem}\label{T:bound-relM1-ori}
Let every solution~$(y,u)$ of system~\eqref{sys-y-opt-s-dyn} satisfy~$\clJ_s(y_s;y,u)\le C_J\norm{y_s}{V'}^2$ with a constant~$C_J=C_J(\beta)>0$ independent of~$y_s$. If~$M_1$ is large enough, then we also have~$\widehat\clJ_s(y_s;y,u)\le\widehat C_J\norm{y_s}{V'}^2$, with a constant~$\widehat C_J=\widehat C_J(\beta)>0$ independent of~$y_s$.
\end{theorem}
\begin{proof}
With~$q\coloneqq P_{\clE_{M_1}^\rmf}y$ and ~$Q\coloneqq y- P_{\clE_{M_1}^\rmf}y$, we find
\begin{align}
 &\dot Q +A Q+(\Id-P_{\clE_{M_1}^\rmf})A_{\rm rc}  Q
 =F \notag
\end{align}
with~$F\coloneqq (\Id-P_{\clE_{M_1}^\rmf})Bu-(\Id-P_{\clE_{M_1}^\rmf})A_{\rm rc}  q$ and
\begin{align}
\tfrac{\rmd}{\rmd t}\norm{Q}{V'}^2 +2\norm{Q}{H}^2 &=-2\langle(\Id-P_{\clE_{M_1}^\rmf})A_{\rm rc}  Q,A^{-1}Q\rangle_{V',V} +2\langle F,A^{-1}Q\rangle_{\rmD(A)',\rmD(A)} \notag\\
  &\le2C_{\rm rc}\norm{Q}{L^2}\norm{Q}{V'}+2\norm{F}{\rmD(A)'}\norm{Q}{H} \notag\\
  & \le\norm{Q}{H}^2+2C_{\rm rc}^2\norm{Q}{V'}^2+2\norm{F}{\rmD(A)'}^2\notag
\end{align}
where~$\norm{F}{\rmD(A)'}\le C_B\norm{u}{\bbR^{M_0}}+C_{\rm rc}\norm{q}{H}$ for a suitable constant~$C_B>0$, which gives us
\begin{align}
\tfrac{\rmd}{\rmd t}\norm{Q}{V'}^2  +\tfrac12\norm{Q}{H}^2
  & \le-\tfrac12\norm{Q}{H}^2+2C_{\rm rc}^2\norm{Q}{V'}^2+4(C_B^2\norm{u}{\bbR^{M_0}}^2+C_{\rm rc}^2\norm{q}{H}^2).\notag
\end{align}
By denoting the eigenvalue of~$A$ defined as
\begin{equation}\notag
\alpha_{M_1+}\coloneqq\min_{Q\in\rmD(A)\bigcap \clE_{M_1}^\perp}\tfrac{\norm{AQ}{H}}{\norm{Q}{H}}=\min_{Q\in\clE_{M_1}^\perp}\tfrac{\norm{Q}{H}^2}{\norm{Q}{V'}^2},
\end{equation}
where~$\clE_{M_1}^\perp$ is the orthogonal space to~$\clE_{M_1}^\perp$ in~$H$, we find that
\begin{align}
\tfrac{\rmd}{\rmd t}\norm{Q}{V'}^2  +\tfrac12\norm{Q}{H}^2
  & \le-(\tfrac12\alpha_{M_1+}-2C_{\rm rc}^2)\norm{Q}{V'}^2+4(C_B^2\norm{u}{\bbR^{M_0}}^2+C_{\rm rc}^2\norm{q}{H}^2).\notag
\end{align}
Now, for~$M_1$ is large enough such that~$\alpha_{M_1+}\ge 4C_{\rm rc}^2$, we obtain
\begin{align}
\tfrac{\rmd}{\rmd t}\norm{Q}{V'}^2  +\tfrac12\norm{Q}{H}^2
  & \le4(C_B^2\norm{u}{\bbR^{M_0}}^2+C_{\rm rc}^2\norm{q}{H}^2), \notag
\end{align}
and time integration gives us for all~$t\ge s\ge 0$,
\begin{align}
\norm{Q(t)}{V'}^2  +\tfrac12\norm{Q}{L^2((s,t),H)}^2
  & \le\norm{Q(s)}{V'}^2+4(C_B^2\norm{u}{L^2((s,t),\bbR^{M_0})}^2+C_{\rm rc}^2\norm{q}{L^2((s,t),H)}^2), \notag\\
  & \le\norm{Q(s)}{V'}^2+8(C_B^2+C_{\rm rc}^2)C_J\norm{y_s}{V'}^2, \notag
 \end{align}
which implies
\begin{align}\notag
\norm{Q}{L^\infty((s,+\infty),V')}^2  +\tfrac12\norm{Q}{L^2((s,+\infty),H)}^2
  & \le(1+8(C_B^2+C_{\rm rc}^2)C_J)\norm{y_s}{V'}^2.
\end{align}
In particular,
\begin{align}
\widehat\clJ_s(y_s;y,u)&=\tfrac12\norm{y}{L^2((s,+\infty),H)}^2+\tfrac12\beta\norm{u}{L^2((s,+\infty),\bbR^{M_0})}^2\notag\\
&=\tfrac12\norm{q}{L^2((s,+\infty),H)}^2+\tfrac12\beta\norm{u}{L^2((s,+\infty),\bbR^{M_0})}^2+\tfrac12\norm{Q}{L^2((s,+\infty),H)}^2\notag\\
&\le(C_J+1+8(C_B^2+C_{\rm rc}^2)C_J)\norm{y_s}{V'}^2,\notag
\end{align}
and the result follows with~$\widehat C_J\coloneqq1+(1+8C_B^2+8C_{\rm rc}^2)C_J$.
\end{proof}

%%%%%%%%%%%%%%%%%%%%%%%%%%
%%%%%%%%%%%%%%%%%%%%%%%%%%
\subsection{The optimal control problem}
Recalling Theorems~\ref{T:main} and~\ref{T:bound-relM1-ori}, for~$M_0$ and~$M_1$ large enough, it is justified to consider the following optimal control problem.
\begin{problem}\label{P:OPT}
For given~$s\ge0$ and~$y_s\in V'$, find a control~$\overline u\in L^2(s+\bbR_+,\bbR^{M_0})$ and corresponding state~$\overline y\in L^2((s,+\infty),H)$ such that
$(y, u)=(\overline y,\overline u)$ satisfies~\eqref{sys-y-opt-s-dyn}, for~$t\ge s$, and minimizes the cost functional~$\clJ_s$ as in~\eqref{clJ}.
 \end{problem}

By the linearity of the dynamics and by the convexity of~$(y,u)\mapsto\clJ_s(y_s;y,u)$ it follows that the optimal pair~$(\overline y,\overline u)=(\overline y,\overline u)(y_s)$ solving Problem~\ref{P:OPT} is unique. The existence of an optimal pair can be shown by a classical  minimizing sequence argument.

We shall  show that the optimal control~$\overline u$ can be taken in feedback form
$\overline u=-B^*\Pi\overline y$, where the operator~$\Pi=\Pi(t)$, can be found by solving a suitable Riccati operator equation, with~$\Pi(t)\in\clL(V',V)$ for almost all~$t>0$.

\begin{assumption}\label{A:boundOPT}
We assume that the optimal pair solving Problem~\ref{P:OPT} satisfies the bound~$\clJ_s(y_s;\overline y,\overline u)\le C_J
\norm{y_s}{V'}^2$ with~$C_J$ independent of~$(s,y_s)$.
\end{assumption}
Observe that Assumption~\ref{A:boundOPT} is satisfied if we take actuators as in Theorem~\ref{T:main}. Indeed, we have~$\clJ_s(y_s;\overline y,\overline u)\le \clJ_s(y_s;y,u)$ where~$(y,u)$ is the state-control pair in Theorem~\ref{T:main}, with~$u_j=-\lambda\langle\fkd_{j},A^{-1}y\rangle$. This theorem gives us~$\norm{P_{\clE_{M_1}^\rmf}y}{L^2((s,+\infty),L^2)}^2\le \norm{y}{L^2((s,+\infty),L^2)}^2\le C_0\norm{y_s}{V'}^2$, with~$C_0$ independent of~$(s,y_s)$. Then, for the control input we also find~$\norm{u}{L^2((s,+\infty),\bbR^{M_\sigma})}^2\le C_1\norm{y_s}{V'}^2$, with~$C_1$ independent of~$(s,y_s)$.

Next, observe that by using the optimal pair~$(\overline y,\overline u)=(\overline y,\overline u)(y_s)$ we can define a semi-scalar product on~$V'$ (cf.~\cite[Ch.~I, Def.~1.1]{Conway90}) as follows
\begin{equation}\notag
0\le ((y_s^1,y_s^2))\coloneqq(P_{\clE_{M_1}^\rmf}\overline y(y_s^1),P_{\clE_{M_1}^\rmf}\overline y(y_s^2))_{L^2((s,+\infty),L^2)}+\beta(\overline u(y_s^1),\overline u(y_s^2))_{L^2((s,+\infty),\bbR^{M_0})}.
\end{equation}

The optimal cost, associated to the initial state~$y_s$, reads~$\clJ_s(y_s;\overline y,\overline u)=\tfrac12((y_s,y_s))$. Furthermore,  for each fixed~$y_s^1$ the mapping~$y_s^2\mapsto \xi_{y_s^1}(y_s^2)\coloneqq((y_s^1,y_s^2))$ is linear and bounded, and $\xi_{y_s^1}\in V=V''=\clL(V',\bbR)$.
Therefore, we have the representation
\begin{equation}\notag
 ((y_s^1,y_s^2))=\langle \Pi_sy_s^1,y_s^2\rangle_{V,V'},\qquad \Pi_sy_s^1\coloneqq\xi_{y_s^1},\qquad \Pi_s\in\clL(V',V).
\end{equation}
Also, note that~$\Pi_s$ is well defined, because if~$\Xi_s$ has the same properties it follows that
$0=\langle (\Pi_s-\Xi_s)y_s^1,y_s^2\rangle_{V,V'}$ for all~$(y_s^1,y_s^2)\in V'\times V'$
which implies that~$\Pi_s-\Xi_s=0$. We can also see that~$\Pi_s$ is linear. Further, it is bounded due to
$\langle \Pi_sy_s^1,y_s^2\rangle_{V,V'}= ((y_s^1,y_s^2))\le 2\clJ_s(y_s^1;\overline y^1,\overline u^1)^\frac12\clJ_s(y_s^2;\overline y^2,\overline u^2)^\frac12\le 2C_J\norm{y_s^1}{V'}\norm{y_s^2}{V'}$, with~$C_J$ as in Assumption~\ref{A:boundOPT}. Thus, with $\Pi_s\in\clL(V',V)$  we have the representation
\begin{equation}\label{Pis-repr}
0\le\tfrac12\langle \Pi_sy_s,y_s\rangle_{V,V'}=\tfrac12((y_s,y_s))=\clJ_s(y_s;\overline y,\overline u)
\end{equation}
for the optimal cost associated with Problem~\ref{P:OPT}.

%%%%%%%%%%%%%%%%%%%%%%%%%%%%
%%%%%%%%%%%%%%%%%%%%%%%%%%%%
\subsection{Dynamic programming principle}\label{sS:adjoint}
Let us consider the following auxiliary optimal control problem.
\begin{problem}\label{P:OPT-aux}
Given~$s\ge0$ and~$y_0\in V'$, find a control~$\underline u\in L^2((0,s),\bbR^{M_0})$ and corresponding state~$\underline y\in L^2((0,s),H)$ such that
$(y, u)=(\underline y,\underline u)$ satisfies~\eqref{sys-y-opt-s-dyn}, for~$t\in(0,s)$, with~$y(0)=y_0$ and minimizes the cost functional
\begin{align}\label{cost-aux}
&\clJ_0^s(y_0;y,u)\coloneqq\tfrac12\norm{P_{\clE_{M_1}^\rmf}y}{L^2((0,s),H)}^2+\beta\tfrac12\norm{u}{L^2((0,s),\bbR^{M_0})}^2+\tfrac12\langle \Pi_sy(s),y(s)\rangle_{V,V'}.
\end{align}
 \end{problem}

Note that~\eqref{cost-aux} includes a cost as~\eqref{clJ}, now in the finite time interval~$(0,s)$, and the optimal cost-to-go $\tfrac12\langle \Pi_sy(s),y(s)\rangle_{V,V'}$ after time~$t=s$ as in~\eqref{Pis-repr}.

Let~$(\overline y_{[0]},\overline u_{[0]})=(\overline y,\overline u)(y_0)$ be the optimal pair associated to Problem~\ref{P:OPT} in the case~$s=0$ with initial state~$y_0$. The dynamic programming principle tells us that this pair is related to the pair solving Problem~\ref{P:OPT-aux} as follows
\begin{equation}\label{DPP}
(\overline y_{[0]},\overline u_{[0]})\rest{(0,s)}=(\underline y,\underline u)(y_0),\qquad (\overline y_{[0]},\overline u_{[0]})\rest{(s,+\infty)}=(\overline y,\overline u)(\underline y(s)),
\end{equation}
where~$(\underline y,\underline u)(\underline y(s))$ is the pair solving Problem~\ref{P:OPT}, with initial state~$y_s=\underline y(s)=\overline y_{[0]}(s)$.

%%%%%%%%%%%%%%%%%%%%%%%%%%%%
%%%%%%%%%%%%%%%%%%%%%%%%%%%%
\subsection{Optimal feedback}\label{sS:feedback}
We show that the optimal control~$\overline u$ solving Problem~\ref{P:OPT} can be taken in feedback form
$\overline u(t)=-B^*\Pi_t\overline y(t)$. For this purpose, we use standard tools from optimal control, namely, the Karush-Kuhn-Tucker conditions associated with Problem~\ref{P:OPT-aux}. For this purpose we define the spaces
\begin{align}
&\clX_0^s\coloneqq W((0,s),L^2,\rmD(A)')\times L^2((0,s),\bbR^{M_0}),\\
&\clY_0^s\coloneqq V'\times L^2((0,s),\rmD(A)')
\intertext{together with the mapping}
&\Psi\colon \clX_0^s\to \clY_0^s,\qquad(y,u)\mapsto (y(0),\dot y +A y+A_{\rm rc}  y -Bu).
\end{align}
Note that~$\Psi$ is surjective, because system~\eqref{sys-y-opt-s-dyn}, with~$h$
in the place of~$Bu$ and with~$s=0$,
\begin{align}\label{sys-y-h}
 &\dot y +A y+A_{\rm rc}  y
 =h,\quad y(0)= y_0,
\end{align}
has a solution ~$y\in\clX_0^s$ (again, we can reason through Galerkin approximations. In particular,
we can take~$u=0$).

By the Karush-Kuhn-Tucker Theorem~\cite[Sect.~A.1]{BarRodShi11}
\cite[Thm.~3.1; Eqs.~(1.1) and~(1.4)]{ZoweKurcyusz79},
we have that there exists~$(\xi,p)\in(\clY_0^s)'=V\times L^2((0,s),\rmD(A))$ such that
\begin{equation}\notag
\rmd\clJ_0^s\rest{(\underline y,\underline u)}=(\xi,p)\circ\Psi,
\end{equation}
which implies that
\begin{align}
&(P_{\clE_{M_1}^\rmf}\underline y,P_{\clE_{M_1}^\rmf}w)_{L^2((0,s),L^2)}+\beta(\underline u,v)_{L^2((0,s),\bbR^{M_0})}+\langle \Pi_s\underline y(s),w(s)\rangle_{V,V'}\notag\\
&\hspace{3em}=\langle w(0),\xi\rangle_{V',V}
+\langle\dot w +A w+A_{\rm rc}  w -Bv,p\rangle_{L^2((0,s),\rmD(A)'),L^2((0,s),\rmD(A))}\notag
\end{align}
for all~$(w,v)\in\clX_0^s$. After standard arguments, for time~$t\in(0,s)$, we find that
\begin{align}
&\dot p-A p-(A_{\rm rc})^*p +P_{\clE_{M_1}^\rmf}\underline y=0,\qquad p(s)= \Pi_s\underline y(s),\notag\\
&\beta\underline u=-B^*p=-B^*\Pi_s\underline y(s).\notag
\end{align}

Recalling~\eqref{DPP}, the dynamics of system~\eqref{sys-y-opt-s-dyn} with the optimal pair~$(y,u)=(\overline y_{[0]},\overline u_{[0]})$ solving Problem~\ref{P:OPT} for time $t\ge0$ and with initial state~$y_0\in V$ reads
\begin{align}\label{sys-y-feed}
&\dot{y} =-Ay-A_{\rm rc}y
 -\beta^{-1}BB^*\Pi y,\qquad y(0)= y_0,
\end{align}
with~$\Pi(t)=\Pi_t$. Further, since~$p\in L^2((0,s),\rmD(A))$, from the relation~$p(s)= \Pi_s\underline y(s)=\Pi_s\overline y_{[0]}(s)$, we conclude that in~\eqref{sys-y-feed} we will have that~$\Pi_t y(t)\in\rmD(A)$, for almost all~$t>0$. In particular, we have that the vector~$B^*\Pi y\in\bbR^{M_\sigma}$, with coordinates~$(B^*\Pi y)_j=\langle \fkd_j,\Pi y\rangle_{\rmD(A)',\rmD(A)}$, is well defined for almost all~$t>0$.

%%%%%%%%%%%%%%%%%%%%%%%%%%%%
%%%%%%%%%%%%%%%%%%%%%%%%%%%%
\subsection{Imposing the exponential stability rate}\label{sS:feedback-mu}

Let us fix~$\mu_{\rm ric}>0$ and introduce the shifted reaction-convection term as
\begin{equation}\notag
A_{\rm rc}^{\mu_{\rm ric}}\coloneqq A_{\rm rc}-\mu_{\rm ric}\Id,\qquad\mu_{\rm ric}>0.
\end{equation}
Since~$A_{\rm rc}^{\mu_{\rm ric}}$ satisfies Assumption~\ref{A:A1} (with a new bound~$C_{\rm rc}^{\mu_{\rm ric}}=C_{\rm rc}+\mu_{\rm ric}\norm{\Id}{\clL(H,V')}$)  we can apply Theorem~\ref{T:main}  (with some~$\mu>0$) to guarantee the existence of a stabilizing control for~$M$ and~$\lambda$  large enough. Then, considering the analogue to Problem~\ref{P:OPT} and following the discussion in previous sections, we arrive at the corresponding cost-to-go
\begin{equation}\label{Pis-repr-mu}
\tfrac12\langle \Pi_s^{\mu_{\rm ric}} \widetilde y_s,\widetilde y_s\rangle_{V,V'}=\clJ_s(\widetilde y_s;\overline y,\overline u),
\end{equation}
with the solution of the analogue of~\eqref{sys-y-feed}
\begin{align}\label{sys-y-feed-mu}
&\dot{\widetilde y} =-A\widetilde y-(A_{\rm rc}-\mu_{\rm ric}\Id)\widetilde y
 -\beta^{-1}BB^*\Pi^{\mu_{\rm ric}} \widetilde y,\qquad y(0)= y_0,
\end{align}
giving us, by taking~$\widetilde u=\beta^{-1}B^*\Pi^{\mu_{\rm ric}} \widetilde y$, the minimizer of the cost functional
\begin{align}\label{sys-y-opt-s-cost-mu}
&\clJ_s(\widetilde y_s;\widetilde y,\widetilde u)=\tfrac12\norm{P_{\clE_{M_1}^\rmf}\widetilde y}{L^2((s,+\infty),H)}^2+\beta\tfrac12\norm{\widetilde u}{L^2((s,+\infty),\bbR^{M_0})}^2.
\end{align}
Now observe that by Theorem~\ref{T:bound-relM1-ori},  $\norm{\widetilde y}{L^2((s,+\infty),H)}^2$ will be bounded for large~$M_1$, and by using Datko Theorem~\cite[Thm.~1]{Datko72} we can derive that~$\norm{\widetilde y}{V'}$ is exponentially decreasing with some small rate~$\mu_0>0$. Therefore, ~$\norm{y}{V'}\coloneqq\rme^{-\mu_{\rm ric}t}\norm{\widetilde y}{V'}$ is exponentially decreasing with rate at least~$\mu_{\rm ric}$ and this~$y$ solves
\begin{align}\label{sys-y-feed-mu-back}
&\dot{y} =-A y-A_{\rm rc} y
 -\beta^{-1}BB^*\Pi^{\mu_{\rm ric}}  y,\qquad y_s= \rme^{-\mu_{\rm ric}s}\widetilde y_s.
\end{align}
Further, note that the minimization of~\eqref{sys-y-opt-s-cost-mu} subject to
\begin{align}\label{sys-y-opt-s-dyn-ric}
 &\dot{\widetilde y} +A\widetilde y+(A_{\rm rc}-\mu_{\rm ric}\Id) \widetilde y
 =B\widetilde u,\quad \widetilde y(s)= \widetilde y_s.
\end{align}
is equivalent  to the minimization of
\begin{align}\label{Jmu-ori}
&\clJ_s^{\mu_{\rm ric}}(y_s; y, u)\coloneqq\tfrac12\norm{\rme^{\mu_{\rm ric}\Bigcdot}P_{\clE_{M_1}^\rmf} y}{L^2((s,+\infty),H)}^2+\beta\tfrac12\norm{\rme^{\mu_{\rm ric}\Bigcdot} u}{L^2((s,+\infty),\bbR^{M_0})}^2,
\end{align}
subject to~\eqref{sys-y-opt-s-dyn}.

%%%%%%%%%%%%%%%%%%%%%%%%%%%%
%%%%%%%%%%%%%%%%%%%%%%%%%%%%
\subsection{Riccati}\label{sS:riccati}
It is well known that with actuators in the pivot space~$H$ the cost-to-go operator~$\Pi^{\mu_{\rm ric}}$ satisfies a Riccati  equation. We confirm here that this is also the case for the ``less regular''  actuators  that we consider, in the space~$\rmD(A)'\supset H$.

From~\eqref{Pis-repr-mu} and~\eqref{sys-y-opt-s-cost-mu}, with~$\overline y(s)=w$, for a generic~$w\in V'$, we find
\begin{align}
\tfrac{\rmd}{\rmd t}\langle \Pi \overline y,\overline y\rangle_{V,V'}\rest{t=s}&=2\tfrac{\rmd}{\rmd t}\clJ_s(w;\overline y,\overline u)=-\norm{P_{\clE_{M_1}^\rmf}w}{H}^2-\beta\norm{\widetilde u(s)}{\bbR^{M_0}}^2\notag\\
&=-(P_{\clE_{M_1}^\rmf}w,w)_{H}-\beta^{-1}(B^*\Pi^{\mu_{\rm ric}} w,B^*\Pi^{\mu_{\rm ric}} w)_{\bbR^{M_0}}\notag\\
&=-\langle P_{\clE_{M_1}^\rmf}w,w\rangle_{V,V'}-\beta^{-1}\langle (B^*\Pi^{\mu_{\rm ric}})^*B^*\Pi^{\mu_{\rm ric}} w, w\rangle_{V,V'}.
\label{Ric1}
\end{align}
Let us denote
\begin{equation}\label{Arc-ric}
L\coloneqq-A-A_{\rm rc}+\mu_{\rm ric}\Id.
\end{equation}
Using the dynamics~\eqref{sys-y-feed-mu}, we expand the left-hand side of~\eqref{Ric1} (proceeding analogously as in~\cite[Sect.~3.2, Rem.~3.11(b)]{BarRodShi11}). Formally, we find
\begin{align}
&\tfrac{\rmd}{\rmd t}\langle \Pi^{\mu_{\rm ric}}\overline  y,\overline y\rangle=\langle \dot\Pi^{\mu_{\rm ric}} \overline y,\overline y\rangle
+\langle \dot{\overline  y}, \Pi^{\mu_{\rm ric}} \overline y\rangle+\langle \Pi^{\mu_{\rm ric}} \overline y,\dot{\overline  y}\rangle\notag\\
&\hspace{3em}=\langle \dot\Pi^{\mu_{\rm ric}} \overline y,\overline y\rangle
+\langle L\overline y
 -\beta^{-1}BB^*\Pi^{\mu_{\rm ric}}\overline  y,\Pi^{\mu_{\rm ric}}\overline  y\rangle+\langle \Pi^{\mu_{\rm ric}}\overline  y,L\overline y
 -\beta^{-1}BB^*\Pi^{\mu_{\rm ric}}\overline  y\rangle,\notag
\end{align}
hence, at $t=s$,
\begin{equation}
\tfrac{\rmd}{\rmd t}\langle \Pi^{\mu_{\rm ric}} \overline y,\overline y\rangle|_{t=s} =\langle (\dot\Pi^{\mu_{\rm ric}} +\Pi^{\mu_{\rm ric}} L+L^*\Pi^{\mu_{\rm ric}} -2\beta^{-1}(B^*\Pi^{\mu_{\rm ric}})^*B^*\Pi^{\mu_{\rm ric}})w,w\rangle.
\label{Ric2}
\end{equation}

Combining~\eqref{Ric1} and~\eqref{Ric2}, we obtain the Riccati equation
\begin{align}
-\dot\Pi= (L^*\Pi^{\mu_{\rm ric}})^*+L^*\Pi^{\mu_{\rm ric}}-\beta^{-1}(B^*\Pi^{\mu_{\rm ric}})^*B^*\Pi^{\mu_{\rm ric}}+P_{\clE_{M_1}^\rmf}.\label{Riccati0}
\end{align}

As we have seen, in~\eqref{sys-y-feed} we will have~$\Pi^{\mu_{\rm ric}}(t) y(t)\in\rmD(A)$ for almost all $t>0$. Let
\begin{equation}\notag
Z(t)\coloneqq \{w\in H\mid \Pi^{\mu_{\rm ric}}(t) w\in\rmD(A)\}.
\end{equation}
We have~$B^*\Pi^{\mu_{\rm ric}}(t)\colon Z(t)\to\bbR^{M_0}$ and~$L^*(t)\Pi^{\mu_{\rm ric}} (t)\colon Z(t)\to H$. Thus~$(L^*\Pi^{\mu_{\rm ric}})^*\colon H\to Z(t)'$ and~$(B^*\Pi^{\mu_{\rm ric}})^*\colon \bbR^{M_0}\to Z(t)'$ are well defined. This implies that all the operators on the right hand side of~\eqref{Riccati0} are well defined from~$Z(t)$ into~$Z(t)'$. Thus also~$\dot\Pi^{\mu_{\rm ric}}$ maps~$Z(t)$ into~$Z(t)'$.
Keeping in mind that we should understand it as~\eqref{Riccati0}, we shall still write the Riccati equation in the canonical form
\begin{align}\label{Riccati}
&\dot\Pi^{\mu_{\rm ric}}+ \Pi^{\mu_{\rm ric}} L+L^*\Pi^{\mu_{\rm ric}} -\Pi^{\mu_{\rm ric}} B_\beta B_\beta^*\Pi^{\mu_{\rm ric}}+C^*C=0
\end{align}
with $B_\beta\coloneqq\beta^{-\frac12}B\quad\mbox{and}\quad C=P_{\clE_{M_1}^\rmf}$.

%%%%%%%%%%%%%%%%%%%%%%%%
%%%%%%%%%%%%%%%%%%%%%%%%
%%%%%%%%%%%%%%%%%%%%%%%%
\section{Application to scalar parabolic equations}\label{S:parab}

We consider system~\eqref{sys-y-parab-intro-expl}, in a rectangular domain~$\Omega=\bigtimes_{n=1}^d(0,l_n)\subset\bbR^d$, $d\in\{1,2,3\}$, with side-lengths~$l_n>0$, $1\le n\le d$,  which we rewrite in evolutionary form as
\begin{subequations}\label{sys-y-parab-abst}
\begin{align}\label{sys-y-parab-abst-dyn}
 &\dot y +A y+A_{\rm rc}  y
 = -\lambda\textstyle\sum\limits_{j=1}^{M_\sigma}\langle\deltafun_{x^j},A^{-1}y(\Bigcdot,t)\rangle_{\rmD(A)',\rmD(A)}\deltafun_{x^j},\qquad y(0)= y_0,
\end{align}
with actuators given by delta distributions~$\deltafun_{x^{M,j}}$, and with
\begin{equation}\label{sys-y-parab-abst-Arc}
A=-\nu\Delta+\Id\quad\mbox{and}\quad
A_{\rm rc}\coloneqq (a-1)\Id+b\cdot\nabla
\end{equation}
\end{subequations}
under either Dirichlet or Neumann boundary conditions~$
\fkB\in\{\Id,\bfn\cdot\nabla\}$,
where~$\bfn=\bfn(\overline x)$ stands for the unit outward normal vector at a point~$\overline x\in\p\Omega$ on the boundary of~$\Omega$.

For brevity, we denote Lebesgue and Sobolev spaces of scalar functions defined in~$\Omega$, as~$L^p\coloneqq L^p(\Omega)$ and~$W^{s,p}\coloneqq W^{s,p}(\Omega)\subset L^p$, for~$s\ge0$ and~$p\ge1$.

 More precisely, we take
\begin{equation}\notag
H\coloneqq L^2,\qquad H'=H,\
\end{equation}
as the pivot Hilbert space, and define the space
\begin{align}
&V=W^{1,2},&\quad&\mbox{for Neumann bcs};\notag\\
&V=\left\{h\in\left. W^{1,2}\,\right|\; h\rest{\p\Omega}=0\right\}&\quad&\mbox{for Dirichlet bcs};\notag
\end{align}
and the operator
\begin{equation}\notag
\langle Ay,w\rangle_{V',V}\coloneqq\nu(\nabla y,\nabla z)_{(L^2)^d}+( y,z)_{L^2}
\end{equation}
with domain
\begin{equation}\notag
\rmD(A)\coloneqq\{h\in H\mid Ah\in H\}=\left\{h\in\left. W^{2,2}\,\right|\; \fkB h\rest{\p\Omega}=0\right\}.
\end{equation}
We endow the spaces~$H$,~$V$ and~$\rmD(A)$ with the scalar products
\begin{equation}\notag
(y,z)_H\coloneqq(y,z)_{L^2},\qquad (y,z)_V\coloneqq\langle Ay,w\rangle_{V',V},
\qquad
(y,z)_{\rmD(A)}\coloneqq(Ay,Aw)_{L^2}.
\end{equation}

For the reaction and convection terms, we assume that
\begin{subequations}\label{assum.ab.parab}
\begin{align}
 & a\in L^\infty(\bbR_+,L^3(\Omega)) \quad\mbox{and}\quad b\in L^\infty(\Omega\times\bbR_+)^d\quad\mbox{with}\quad \nabla\cdot b\in L^\infty(\bbR_+,L^3(\Omega)),
\end{align}
and  further
\begin{align}
b\cdot\bfn=0,\quad\mbox{if}\quad\fkB=\bfn\cdot\nabla.
\end{align}
\end{subequations}

%%%%%%%%%%%%%%%%%%%%%%%%%
%%%%%%%%%%%%%%%%%%%%%%%%%
\subsection{Satisfiability of Assumptions~\ref{A:A0sp}--\ref{A:A1}}\label{sS:parabAssumFree}
It is clear that Assumptions~\ref{A:A0sp}--\ref{A:A0cdc} are satisfied. We show now that Assumption~\ref{A:A1} is a corollary of~\eqref{assum.ab.parab}.
Indeed, for arbitrary~$w\in V$ and~$y\in V$, we find that
\begin{align}
&\langle (a(t,\Bigcdot)-1)y+b(t,\Bigcdot)\cdot\nabla y,w\rangle_{V',V}=((a(t,\Bigcdot)-1)y+b(t,\Bigcdot)\cdot\nabla y,w)_{L^2}\notag\\
&\hspace{2em}\le\norm{a(t,\Bigcdot)-1}{L^3}\norm{y}{L^2}\norm{w}{L^6}-(y,\nabla \cdot(b(t,\Bigcdot)w))_{L^2}+\int_{\p\Omega}ywb\cdot\bfn\,\rmd\p\Omega\notag\\
&\hspace{2em}\le C_1\norm{a(t,\Bigcdot)-1}{L^3}\norm{y}{L^2}\norm{w}{W^{1,2}}
+\norm{b(t,\Bigcdot)}{(L^\infty)^d}\norm{y}{L^2}\norm{\nabla w}{(L^2)^d}\notag\\
&\hspace{3em}+\norm{\nabla\cdot b(t,\Bigcdot)}{L^3}\norm{y}{L^2}\norm{w}{L^6}\notag\\
&\hspace{2em}\le C_2(\norm{a(t,\Bigcdot)-1}{L^3}+\norm{b(t,\Bigcdot)}{(L^\infty)^d}+\norm{\nabla\cdot b(t,\Bigcdot)}{L^3}) \norm{y}{L^2}\norm{w}{V}\notag.
\end{align}
Therefore, we can see that Assumption~\ref{A:A1} holds with
\begin{equation}\notag
C_{\rm rc}=C_2(\norm{a-1}{L^\infty(\bbR_+,L^3)}+\norm{b(t,\Bigcdot)}{(L^\infty(\Omega\times\bbR_+))^d}+\norm{\nabla\cdot b(t,\Bigcdot)}{L^\infty(\bbR_+,L^3)}),
\end{equation}
by a continuity argument and density of~$V\subset L^2$.

%%%%%%%%%%%%%%%%%%%%%%%%%
%%%%%%%%%%%%%%%%%%%%%%%%%
\subsection{Satisfiability of Assumption~\ref{A:poincare}}\label{sS:actuators-parab}
We construct, for each~$M\in\bbN_+$, the supports~$x^{M,j}\in\Omega=\bigtimes\limits_{n=1}^d(0, l_n)$, $1\le j\le M_\sigma$, of~$M_\sigma=(d+1)M^d$ delta distributions~$\deltafun_{x^{M,j}}$.

For~$M=1$, in the case~$d=1$ we take two distinct  points in~$\Omega$; in the case~$d=2$ we take~$3$ noncolinear points in~$\Omega$; in the case~$d=3$ we take~$4$ noncoplanar  points in~$\Omega$. Let us denote those points as
\begin{subequations}\label{UM-expl-parab}
\begin{equation}\label{UM-refM=1-parab}
x^{1,j}\in\Omega,\qquad 1\le j\le d+1.
\end{equation}

For~$M>1$, we partition the rectangle~$\Omega$ into~$M^d$ rescaled copies~$\Omega^{M,k}$, $1\le k\le M^d$, of~$\Omega$.
In each copy~$\Omega^{M,k}=v^k+\frac1{M}\Omega$, $v^k\in\{v=(v_1,\dots,v_d)\in\bbR^d\mid v_n=(i-1)\frac{ l_n}{M},\; 1\le i\le M,\; 1\le n\le d\}$, we
select~$d+1$ points
\begin{equation}\label{UM-M>1-parab}
x^{M,(d+1)(k-1)+m}=v^k+\frac1{M}x^{1,m}\in\Omega^{M,k},\qquad 1\le m\le d+1.
\end{equation}
In this way, ~$\Omega^{M,k}\setminus\{x^{M,(d+1)(k-1)+m}\mid1\le m\le d+1\}=v^k+\frac1{M}\Omega\setminus\{x^{1,m}\mid1\le m\le d+1\}$.

Thus we arrive at the set of actuators, for a fixed~$M\in\bbN_+$, given by
\begin{equation}\label{UM-constr-parab}
U_M\coloneqq\{\deltafun_{x^{M,j}}\mid 1\le j\le M_\sigma\},\qquad\clU_M\coloneqq \linspan U_M.
\end{equation}
\end{subequations}

Now let us fix an arbitrary set of auxiliary functions
 \begin{subequations}\label{auxtilUM-parab}
 \begin{align}
&\widetilde U_M\coloneqq\{\Psi^{M,j}\mid 1\le j\le M_\sigma\}\subset\rmD(A),\qquad\widetilde \clU_M\coloneqq \linspan\widetilde U_M
\intertext{satisfying}
&\Psi^{M,j}(x^{M,j})=1\quad\mbox{and}\quad\Psi^{M,j}(x^{M,i})=0\quad\mbox{for}\quad i\ne j.
\end{align}
\end{subequations}
Finally, we denote the subspace
\begin{align}\label{DM-constr-parab}
\rmD_M\coloneqq\{z\in\rmD(A)\mid z(x^{M,i})=0\mbox{ for all }1\le i\le M_\sigma\}.
\end{align}

Note that, for the delta distribution actuators~$\fkd^{M,j}\coloneqq\deltafun_{x^{M,j}}$, we have the relation~$\langle\fkd^{M,j},z\rangle= z(x^{M,i})$. Therefore, to show that
 Assumption~\ref{A:poincare} holds true it remains to show that the constant, in~\eqref{Poinc_const},
 \begin{align}\notag
\xi_{M_+}=\inf_{\varTheta\in\rmD_M\setminus\{0\}}
\tfrac{\norm{\varTheta}{\rmD(A)}^2}{\norm{\varTheta}{V}^2},
\end{align}
diverges to~$+\infty$ as $M$ increases. For this purpose, we  follow a slight variation of the arguments in~\cite[Sect.~4.2]{Rod21-jnls}.

Firstly, we observe that, due to the choice of the~$x^{1,j}$, in the case~$M=1$ the function~$q\mapsto\norm{(q(x^{1,1}),\dots,q(x^{1,d+1}))}{\bbR^{d+1}}$ is a seminorm on the set of polynomials of degree less or equal to~$1$. Therefore, we can follow the steps as in~\cite[Sect.~4.2]{Rod21-jnls} to conclude that~$\norm{\nabla\nabla z}{(L^2)^{d^2}}$ is a norm equivalent to the usual~$W^{2,2}$-norm in the space~$\rmD_1$ in~\eqref{DM-constr-parab}. In particular~$\norm{\nabla\nabla z}{(L^2)^{d^2}}^2\ge C_0\norm{z}{W^{1,2}}^2$ for all~$z\in\rmD_1$.

 Next, again following~\cite[Sect.~4.2]{Rod21-jnls}, using the fact that for~$M>1$ the domain~$\Omega$ is partitioned in rescaled (congruent) copies of itself, we can conclude that
   \begin{equation}\notag
 \norm{\nabla\nabla z}{(L^2)^{d^2}}^2\ge C_1M^2\norm{z}{W^{1,2}}^2,\quad\mbox{for all}\quad z\in\rmD_M.
 \end{equation}
 Now, using also the equivalence of the norms~$\rmD(A)$ and~$W^{2,2}$ in~$\rmD(A)\subseteq W^{2,2}$ and the equivalence of the norms of~$V$ and~$W^{1,2}$ in~$V\subseteq W^{1,2}$, we obtain for an appropriate constant~$C_2>0$,
  \begin{align}\notag
\lim_{M\to+\infty}\xi_{M_+}&\ge\lim_{M\to+\infty}\inf_{\varTheta\in\rmD_M\setminus\{0\}}
C_2\tfrac{\norm{\varTheta}{W^{2,2}}^2}{\norm{\varTheta}{W^{1,2}}^2}\ge C_2\lim_{M\to+\infty}\inf_{\varTheta\in\rmD_M\setminus\{0\}}
\tfrac{\norm{\nabla\nabla\varTheta}{(L^{2})^{d^2}}^2}{\norm{\varTheta}{W^{1,2}}^2}\\
&\ge C_2C_1\lim_{M\to+\infty}M^2=+\infty.\notag
\end{align}
  Therefore, Assumption~\ref{A:poincare} holds true.

We have shown that Assumptions~\ref{A:A0sp}--\ref{A:poincare} are satisfiable. Thus the abstract result in Theorem~\ref{T:main} can be applied to obtain the following result.
\begin{corollary}%
 Let~$\mu>0$ be given. If we construct the actuators locations~$x^{M,j}$ as in Section~\ref{sS:actuators-parab} and choose the positive integer~$M$ and the constant~$\lambda>0$  large enough, then
the solution of~\eqref{sys-y-parab-intro-expl} satisfies, under Dirichlet or Neumann bcs, the estimate
\begin{equation}\notag
\norm{y(t)}{V'}\le\ex^{-\mu (t-s)}\norm{y(s)}{V'},\quad\mbox{for all   }t\ge s\ge0\mbox{ and all }y_0\in V'.
\end{equation}
\end{corollary}

 \begin{remark}\label{R:ext-polygonal}
 We applied the strategy above to a rectangular (box) domain~$\Omega\subset\bbR^d$.
  The strategy can be applied to more general convex polygonal domains~$\bfP\subset\bbR^2$. Indeed, the strategy is based on the fact that a rectangular domain~$\bfR\subset\bbR^2$ can be decomposed into $M^2$ similar  re-scaled copies of itself. Now, we recall that a triangle~$\bfT\subset\bbR^2$ can also be (up to a rotation) decomposed into~$4^{M-1}$ similar re-scaled copies of itself. Thus, the arguments can be applied to triangular domains. Hence, after a triangulation, the strategy can also be applied to general convex polygonal domains~$\bfP\subset\bbR^2$. Now using these 2D polygonal domains, we can see that we can apply the strategy to cylindrical domains as~$\bfC=\bfP\times (0,L)\subset\bbR^3$.
The extension/applicability  of the strategy to more general 3D domains is not clear yet and may require extra work.
\end{remark}

%%%%%%%%%%%%%%%%%%%%%%%%%%%%
%%%%%%%%%%%%%%%%%%%%%%%%%%%%
%%%%%%%%%%%%%%%%%%%%%%%%%%%%
\section{Discretization of the Riccati equation}\label{S:discRicc}
We briefly address the discretization of the Riccati equation in a  piecewise linear finite-elements context.
For a more general context we refer the reader to~\cite{MalqvistPerStill18} and references therein.

Let $Q\in \mathcal{L}(X,X')$, with $X$ a Hilbert space be given.
We assume that~$V=W^{1,2}(\Omega)$ is dense in~$X$ and approximate elements in~$X$ by elements in the finite-dimensional space
\begin{equation}\notag
H_N\coloneqq\linspan\{\bfh_n\mid 1\le n\le N\}
\end{equation}
where the~$\bfh_n$ are the (piecewise linear) hat functions, forming a basis for the finite-element space~$H_N$, defined by
\begin{equation}\notag
\bfh_n(p_n)=1\quad\mbox{and}\quad\bfh_n(p_m)=0,\quad\mbox{for}\quad 1\le n,m\le N,\quad n\ne m,
\end{equation}
where the~$p_n$ are the points in a given triangulation~$\clT$ of~$\overline\Omega$.
Since elements in~$X$ are approximated by elements in~$H_N$, we consider now an arbitrary pair
\begin{equation}\label{yzinHN}
(y,z)\in H_N\times H_N.
\end{equation}

Let  $[Q]\in\bbR^{N\times N}$ be the matrix with entries in the $i$th row and $j$th column given by
\begin{equation}\notag
[Q]_{ij}\coloneqq\langle Q h_j,h_i\rangle_{X',X}.
\end{equation}
Note that, $\langle Q y, z\rangle_{X',X}\eqqcolon\overline z^\top[Q] \overline y$,
where $\overline v\in\bbR^{N\times1}$ is the column vector whose entries~$\overline v_{n,1}$, for~$1\le n\le N$, are the finite-element coordinates of~$v=\bfh(\overline v)\coloneqq\sum\limits_{n=1}^N\overline v_{n,1}\bfh_n$. Next, let us denote
\begin{equation}\label{overQ}
\overline Q \coloneqq\bfM^{-1}[Q]\in\bbR^{N\times N}
\end{equation}
where~$\bfM=[\Id]$ is the mass matrix, $(y, z)_{L^2(\Omega)}=\overline z^\top\bfM\overline y$, for~$(y,z)$ as in~\eqref{yzinHN}.

For a given~$y\in H_N$, let us define the unique vector~$\overline{Qy}\in\bbR^{N\times 1}$ defined by
\begin{equation}\notag
\langle Q y, z\rangle_{X',X}\eqqcolon \overline z^\top\bfM \overline{Qy}\quad\mbox{for all}\quad
 z\in H_N.
\end{equation}
Then, we can write~$\langle Q y, z\rangle_{X',X}=\overline z^\top[Q] \overline y=\overline z^\top\bfM\overline Q \overline y$, which gives us
\begin{align}\label{overQsplit}
\overline{Qy}=\overline Q \overline y,\quad\mbox{for all}\quad \overline y=\in\bbR^{N\times 1}\quad\mbox{with}\quad y=\bfh(\overline y)\in H_N.
\end{align}
Further, from~$\langle Q^*  y, z\rangle_{X',X}=\langle Q z, y\rangle_{X',X}
=  \overline y^\top [Q] \overline z=\overline z^\top[Q]^\top\overline y$, we find
\begin{equation}\label{Q*}
[Q]^\top=[Q^*].
\end{equation}

Now, for a given product $L_1L_2\in\clL(X,X')$ with~$L_2\in\clL(X,X_1')$  and~$L_1\in\clL(X_1',X')$, for another Hilbert spaces~$X_1$, and with dense inclusions~$V\subset X_1$ and~$V\subset X_1'$, we write
\begin{align}
\overline z^\top [L_1L_2] \overline y&=\langle L_1L_2y,z\rangle_{X',X}=\langle L_2y,L_1^*z\rangle_{X_1',X_1}\notag\\
&=\overline {L_1^*{z}}^\top \bfM \overline{L_2} \overline y=\overline{z}^\top\overline {L_1^*}^\top \bfM \overline{L_2} \overline y
=\overline{z}^\top(\bfM^{-1}[L_1^*])^\top\bfM \overline{L_2} \overline y\notag\\
&=\overline{z}^\top[L_1^*]^\top \overline{L_2} \overline y=\overline{z}^\top\bfM\overline{L_1}\, \overline{L_2} \overline y\notag
\end{align}
which gives us
\begin{align}\label{L1L2}
[L_1L_2]=\bfM\overline{L_1}\, \overline{L_2}\quad\mbox{and}\quad\overline{L_1L_2}=\overline{L_1}\, \overline{L_2}.
\end{align}

Taking~$X=V'$ and denoting by~$\mathbf\Pi=[\Pi^{\mu_{\rm ric}}]$ the finite-element matrix
associated with the solution of the Riccati equation~\eqref{Riccati}, we find
\begin{align}
0&=[\dot\Pi^{\mu_{\rm ric}}]+ [\Pi^{\mu_{\rm ric}} L]+([\Pi^{\mu_{\rm ric}} L])^\top -[\Pi^{\mu_{\rm ric}} B_\beta B_\beta^*\Pi^{\mu_{\rm ric}}]+[C^*C]\notag\\
&=\dot{\mathbf\Pi}+ \bfM\overline{\Pi^{\mu_{\rm ric}} L}+(\bfM\overline{\Pi^{\mu_{\rm ric}} L})^\top -\bfM\overline{\Pi^{\mu_{\rm ric}} B_\beta B_\beta^*\Pi^{\mu_{\rm ric}}}+\bfM\overline{C^*C}\notag\\
&=\dot{\mathbf\Pi}+ \mathbf\Pi\overline{L}+\overline{L}^\top\mathbf\Pi -\mathbf\Pi\overline{B_\beta B_\beta^*}\bfM^{-1}\mathbf\Pi+\bfM\overline{C^*C}\notag\\
&=\dot{\mathbf\Pi}+ \mathbf\Pi\overline{L}+\overline{L}^\top\mathbf\Pi -\mathbf\Pi\bfM^{-1}[B_\beta B_\beta^*]\bfM^{-1}\mathbf\Pi+\bfM\overline{C^*C}\notag
\end{align}
 Recalling~$B$ in~\eqref{sys-y-opt-s-dyn}, we note that
 $\langle B_\beta B_\beta^*y,z\rangle_{\rmD(A)',\rmD(A)}=(B_\beta^*y,B_\beta^* z)_{\bbR^{M_0}}$. Hence, with
\begin{align}\notag
\bfB_\beta\coloneqq\beta^{-\frac12}[b_{(n,j)}]\in\bbR^{N\times M_0},
\end{align}
where~$[b_{(n,j)}]$ is the matrix with columns containing the finite-element vectors~$b_{(:,j)}$ corresponding to the delta actuators~$\delta_{x^j}$, $1\le j\le M_0$, we obtain
 \begin{equation}\notag
\langle B_\beta B_\beta^*y,z\rangle_{\rmD(A)',\rmD(A)}=(B_\beta^*y,B_\beta^* z)_{\bbR^{M_0}} =(\bfB_\beta^\top \overline z)^\top\bfB_\beta^\top \overline y=\overline z^\top\bfB_\beta\bfB_\beta^\top \overline y.
\end{equation}
Thus, we find
$[B_\beta B_\beta^*]=\bfB_\beta\bfB_\beta^\top$.%

Next for $C=C^*=C^*C=P_{\clE_{M_1}^\rmf}$, the  orthogonal projection onto~$\clE_{M_1}^\rmf$, we recall that
\begin{equation}\notag
\overline{P_{\clE_{M_1}^\rmf}y}=[c_{(n,j)}]\mathbf\Theta\,[c_{(n,j)}]^\top\bfM\overline y=[c_{(n,j)}] \mathbf\Theta_\bfc^\top\mathbf\Theta_\bfc\,[c_{(n,j)}]^\top\bfM\overline y
\end{equation}
where the $j$th column of~$[c_{(n,j)}]$ contains the finite-element vector corresponding to the eigenfunction~$e_j$ of~$A$, $1\le j\le M_1$, and~$\mathbf\Theta_\bfc$ is the Cholesky factor of~$\mathbf\Theta\coloneqq([c_{(n,j)}]^\top\bfM [c_{(n,j)}])^{-1}$.

Therefore, we arrive at the discretization of~\eqref{Riccati} as
\begin{subequations}\label{Riccati-Disc}
\begin{align}
&0=\dot{\mathbf\Pi}+ \mathbf\Pi\overline{L}+\overline{L}^\top\mathbf\Pi -\mathbf\Pi\fkB_\beta \fkB_\beta^\top\mathbf\Pi+\fkC^\top\fkC
\intertext{with}
&\overline{L}=\bfM^{-1}[L^{\mu_{\rm ric}}],\quad\fkB_\beta=\bfM^{-1}\bfB_\beta\quad\mbox{and}\quad\fkC=\mathbf\Theta_\bfc[c_{(n,j)}]^\top\bfM,
\end{align}
\end{subequations}
where~$[L^{\mu_{\rm ric}}]$ is  the matrix corresponding to a spatial finite-dimensional discretization of the
diffusion-reaction-convection terms. That is, with~$(y,z)$ as in~\eqref{yzinHN},
$(Ly,z)_{V',V}= \overline z^\top [L^{\mu_{\rm ric}}]\overline y$. Namely, for the diffusion, reaction, and convection terms we take
\begin{align}
\langle (-\nu\Delta +\Id)y, z\rangle_{V',V}&=\overline z^\top(\nu\bfS+\bfM)\overline y,\notag\\
\langle (a-1-\mu_{\rm ric}) y, z\rangle_{V,V'}&=\overline z^\top[\tfrac12(\clD_{[\overline{ a-1-\mu_{\rm ric}}]}\bfM+\bfM\clD_{[\overline{a-1-\mu_{\rm ric}}]}\overline y,\notag\\
\langle b\cdot\nabla y, z\rangle_{X',X}&=\overline z^\top(\clD_{[\overline b_1]}\bfD_{x_1}+(\clD_{[\overline b_2]}\bfD_{x_2})\overline y,\notag
\end{align}
where~$\bfS$ is the stiffness matrix, $\clD_{[\overline v]}$ is the diagonal matrix with entries given by the vector~$\overline v$, and~$\bfD_{x_i}$ is the finite-element matrix corresponding to the discretization of the spatial partial derivative~$\frac{\p}{\p x_i}$, that is,  $\langle \frac{\p}{\p x_i}y, z\rangle_{V',V} = \overline z^\top\bfD_{x_i}\overline y$, for~$(y,z)$ as in~\eqref{yzinHN}.

%%%%%%%%%%%%%%%%%%%%%%%%
%%%%%%%%%%%%%%%%%%%%%%%%
%%%%%%%%%%%%%%%%%%%%%%%%
\section{Numerical results}\label{S:numRes}
We consider actuators constructed as in Section~\ref{sS:actuators-parab}.
The results of the following simulations utilize a spatial discretization based on piecewise linear finite elements (hat functions) associated to an unstructured triangulation~$\clT$ of the spatial rectangular domain and
to a temporal discretization based on an implicit-explicit scheme.  An implicit Crank--Nicolson scheme is used to discretize the symmetric diffusion and reaction terms and  an explicit Adams--Bashford extrapolation is used to discretize the nonsymmetric convection term and the feedback control term.

%%%%%%%%%%%%%%%%%%%%%%%%
%%%%%%%%%%%%%%%%%%%%%%%%
\subsection{Explicit feedback control}\label{sS:numResExpl}
 We present the results of simulations showing the stability of system~\eqref{sys-y-parab-intro-expl} under Neumann boundary conditions,
 \begin{subequations}\label{sys-y-parab-num-expl}
\begin{align}
 &\tfrac{\p}{\p t} y +(-\nu\Delta +\Id)y+(a-1)y +b\cdot\nabla y
 =\textstyle\sum\limits_{j=1}^{M_0}K^{M,\lambda}_j(y)\deltafun_{x^{M,j}},\\
  &(\bfn\cdot\nabla y)\rest{\p\Omega}=0,\quad
y(0)=y_0,
       \end{align}
with explicit feedback control
\begin{equation}
K^{M,\lambda}=(K^{M,\lambda}_1,\dots,K^{M,\lambda}_{M_0}),\qquad K^{M,\lambda}_j(y)\coloneqq -\lambda (-\nu\Delta +\Id)^{-1}y)(x^{M,j}).
\end{equation}
 \end{subequations}

We shall use the two configurations of the actuators shown as shown in Fig.~\ref{fig:mesh0}, corresponding to the cases~$M\in\{1,2\}$, following the construction in Section~\ref{sS:actuators-parab}, by taking~$M_0=M_\sigma\in\{3,12\}$ actuators.
In the same figure we can see the (coarsest) triangulation~$\clT=\clT^0$ that we shall consider. The  locations of the actuators do not coincide with mesh points. Thus, for a fixed~$M$, we shall write each location~$x^{M,j}$ as a convex combination of the vertices $p^1(T_k)$, $p^2(T_k)$, $p^3(T_k)$, of a (closed) triangle~$T_k$ containing~$x^{M,j}$, that is,
\begin{equation}\notag
x^{M,j}=\textstyle\sum\limits_{n=1}^{3}c_np^n(T_k),\qquad \textstyle\sum\limits_{n=1}^{3}c_n=1,\quad 0\le c_n\le 1.
\end{equation}
Then we approximate the actuator~$\deltafun_{x^{M,j}}$ as
\begin{equation}\label{delta-convex}
\deltafun_{x^{M,j}}\approx\deltafun_{x^{M,j}}^{\clT}=\textstyle\sum\limits_{n=1}^{3}c_n\deltafun_{p^n(T_k)},
\end{equation}
where the superscript in~$\fkd_j^{\clT}$ underlines that the approximation depends on the triangulation~$\clT$ of the spatial domain. Note that in the case that there is more than one triangle in~$\clT$ containing~$x^{M,j}$, the approximation~$\deltafun_{x^{M,j}}^{\clT}$ does not depend on the choice of the triangle~$T_k\in\clT$ containing~$x^{M,j}$.
\begin{figure}[ht]
\centering
\subfigure[Case~$M=1$.\label{fig:mesh0M1}]
{\includegraphics[width=0.45\textwidth,height=0.31\textwidth]{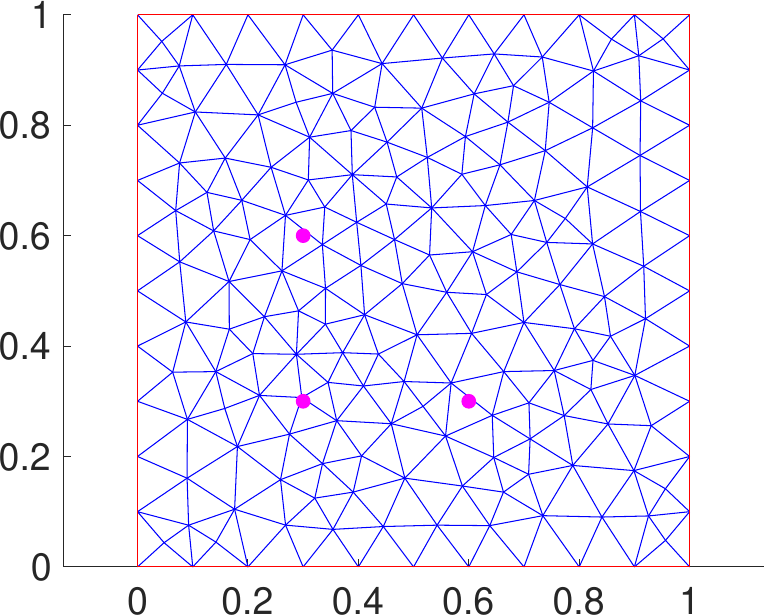}}
\quad
\subfigure[Case~$M=2$.\label{fig:mesh0M2}]
{\includegraphics[width=0.45\textwidth,height=0.31\textwidth]{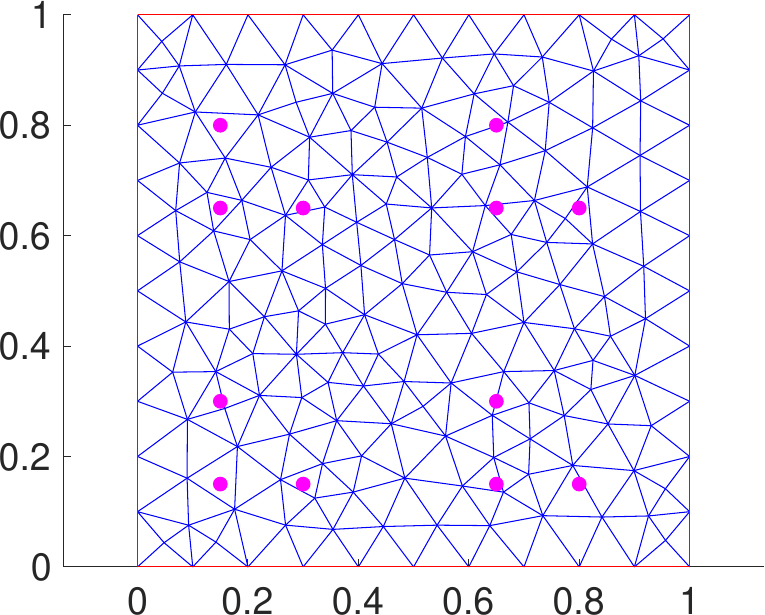}}
\caption{Initial triangulation~$\clT^0$ and actuators locations~$x^{M,j}$.}
\label{fig:mesh0}
\end{figure}

As reaction and convection coefficients in system~\eqref{sys-y-parab-num-expl} we take
\begin{subequations}\label{param.numExpl}
\begin{align}
&a(t,x)-1=-3-(2-x_1)\cos(\pi x_2)-2\norm{\sin(t+x_2)}{\bbR},\\
&b(t,x)=\begin{bmatrix}
\frac{t+2}{t+1}x_1(x_1-1)x_2\\
-\cos(t)(x_1-\frac12)x_2(x_2-1)
\end{bmatrix},
\intertext{and as diffusion coefficient and initial state we take}
&\nu=0.1,\quad\mbox{and}\quad y_0=x_1(1+\sin(2x_2)).
\end{align}
\end{subequations}

In Fig.~\ref{fig:ErrorVp_3act} we show results corresponding to~$3$ actuators located as in Fig.~\ref{fig:mesh0M1}, corresponding to the case~$M=1$, using the time-step~$t^{\rm step}=10^{-4}$ for temporal discretization.  Fig.~\ref{fig:ErrorVp_3act} shows that the free dynamics, corresponding to the case~$\lambda=0$, is unstable in the $V'$~norm. Recall that our stabilizability result has been derived in such distribution space norm. We can see that such~$3$ actuators are not able to stabilize the system for feedback gain parameters $\lambda\in\{10,50,100\}$. Further, we observe a convergence-like  behavior to an unstable regime, as~$\lambda$ increases, of the norm of the solution, which allows us to extrapolate that such actuators are unable to stabilize the system, no matter how large we take~$\lambda$. This confirms the theoretical result on the necessity of taking~$M$  large enough. We show in Fig.~\ref{fig:ErrorLinf_3act} also the $L^\infty$-norm of the solution, thus showing the behavior of the largest magnitude reached by the solution.

\begin{figure}[ht]
\centering
\subfigure[Evolution of~$V'$-norm.\label{fig:ErrorVp_3act}]
{\includegraphics[width=0.45\textwidth,height=0.31\textwidth]{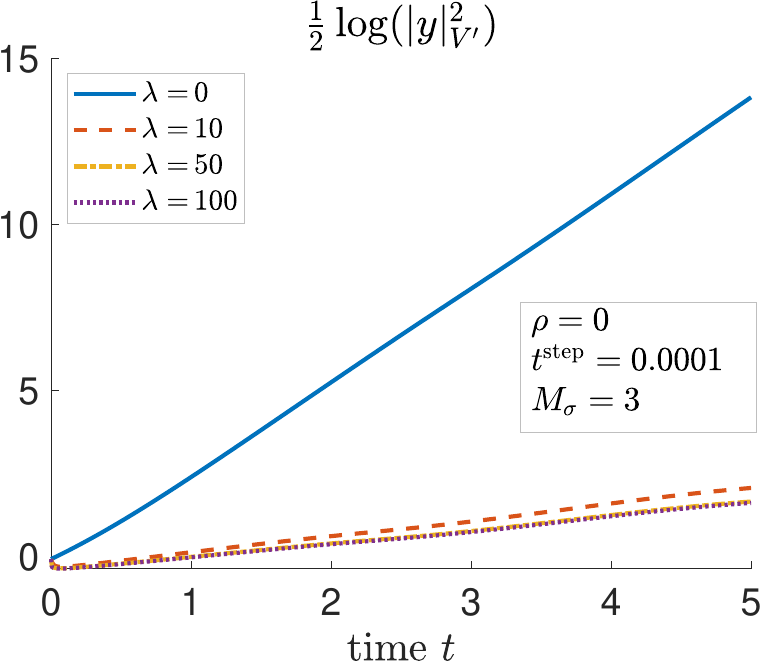}}
\quad
\subfigure[Evolution of~$L^\infty$-norm.\label{fig:ErrorLinf_3act}]
{\includegraphics[width=0.45\textwidth,height=0.31\textwidth]{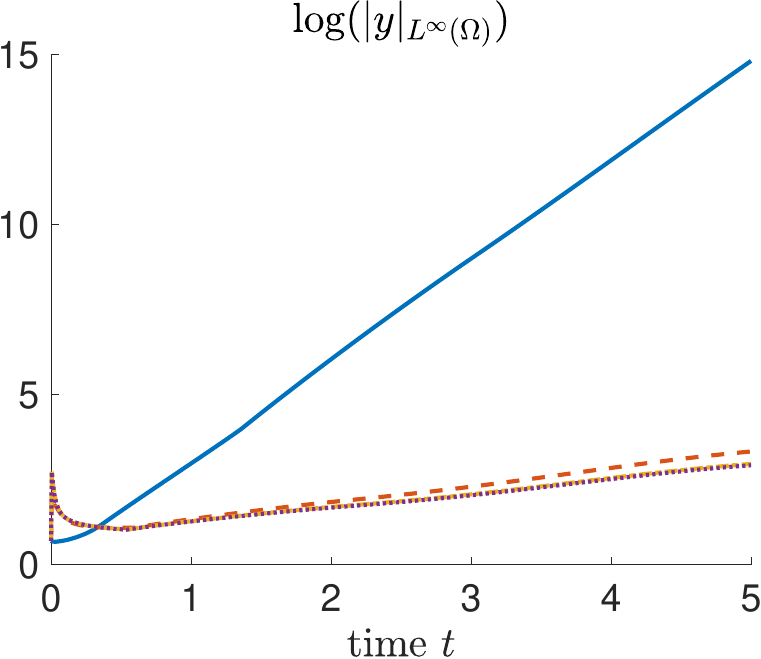}}
\caption{With explicit feedback; $M=1$. Data~\eqref{param.numExpl}.}
\label{fig:Error_3act}
\end{figure}

Next, we increase~$M$, by taking~$M=2$, corresponding to considering a~$2\times2$ partition of the spatial domain into~$4$ rescaled copies of itself and by taking~$12$ actuators located as in Fig.~\ref{fig:mesh0M2}. In addition, we shall also consider~$3$ consecutive refinements of the spatial triangulation~$\clT^0$ in Fig.~\ref{fig:mesh0M2}. Hence, we will have~$4$ triangulations as
\begin{equation}\notag
\clT^\rho,\qquad \rho\in\{0,1,2,3\},
\end{equation}
where~$\clT^{\rho+1}$ is obtained from~$\clT^\rho$ by dividing each triangle~$T^\rho_k\in\clT^\rho$ into~$4$ congruent triangles by connecting the middle points of the edges of~$T^\rho_k$ (regular refinement). The initial mesh~$\clT^0$ was generated with the Matlab routine~{\tt initmesh}, and each refimement was done using the Matlab routine~{\tt refinemesh}.

\begin{figure}[ht]
\centering
\subfigure%[]
{\includegraphics[width=0.45\textwidth,height=0.31\textwidth]{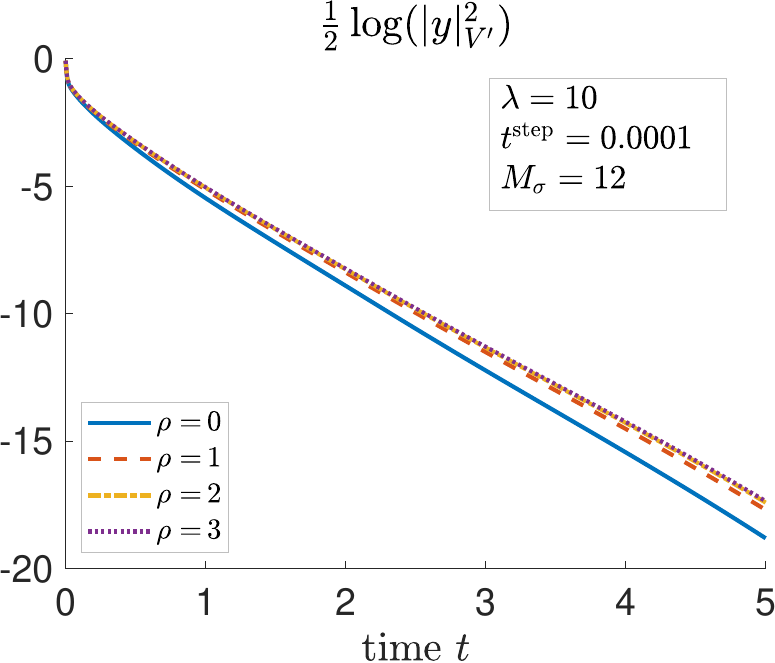}}
\quad
\subfigure%[]
{\includegraphics[width=0.45\textwidth,height=0.31\textwidth]{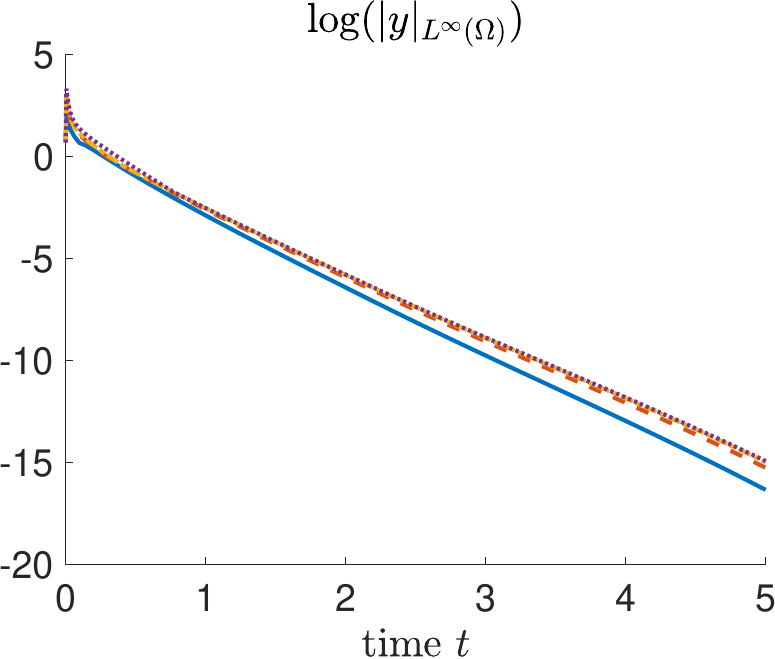}}
\caption{With explicit feedback; $M=2$. Spatial refinements. Data~\eqref{param.numExpl}.}
\label{fig:Error_12act}
\end{figure}
Fig.~\ref{fig:Error_12act} shows that delta distributions actuators located as in Fig.~\ref{fig:mesh0M2} are able to stabilize the system (for the given initial state). We also see that as we refine the spatial mesh we observe a convergence of the evolution of both the $V'$-norm and the $L^\infty$-norm to a stable behavior. With respect to the~$V'$-norm, this confirms our theoretical stability result. With respect to the~$L^\infty$-norm, we do not have such a stability result, but we can observe that this norm also converges exponentially to zero.

We show time snapshots of the solutions in Fig.~\ref{fig:timesnap-rho}, where we can see the effect of the feedback control action since the sign of the solution at actuators locations tends to be opposite to the sign of the solution at neighboring points.
\begin{figure}[ht]
\centering
\subfigure[$\rho=0$]
{\includegraphics[width=0.32\textwidth,height=0.28\textwidth]{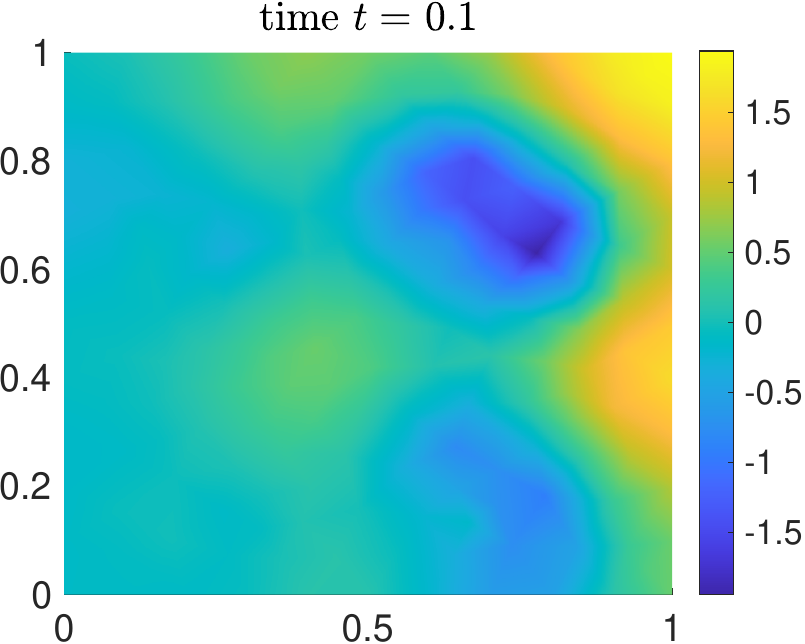}}
\subfigure[$\rho=0$]
{\includegraphics[width=0.32\textwidth,height=0.28\textwidth]{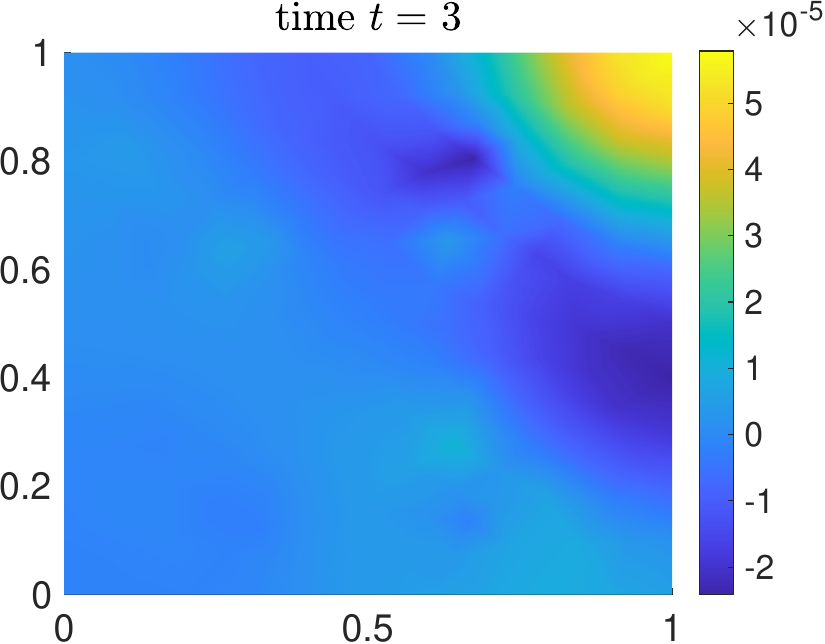}}
\subfigure[$\rho=0$]
{\includegraphics[width=0.32\textwidth,height=0.28\textwidth]{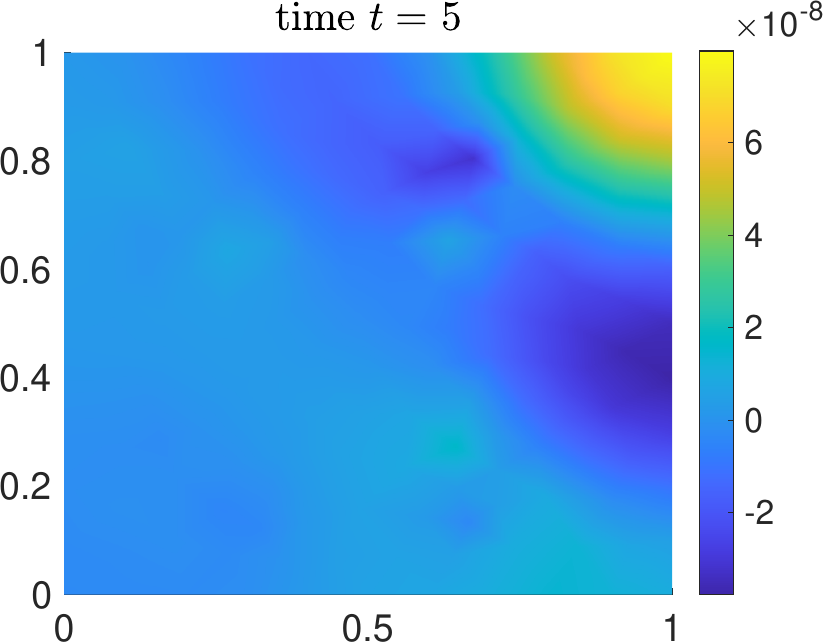}}
\\
\subfigure[$\rho=1$]
{\includegraphics[width=0.32\textwidth,height=0.28\textwidth]{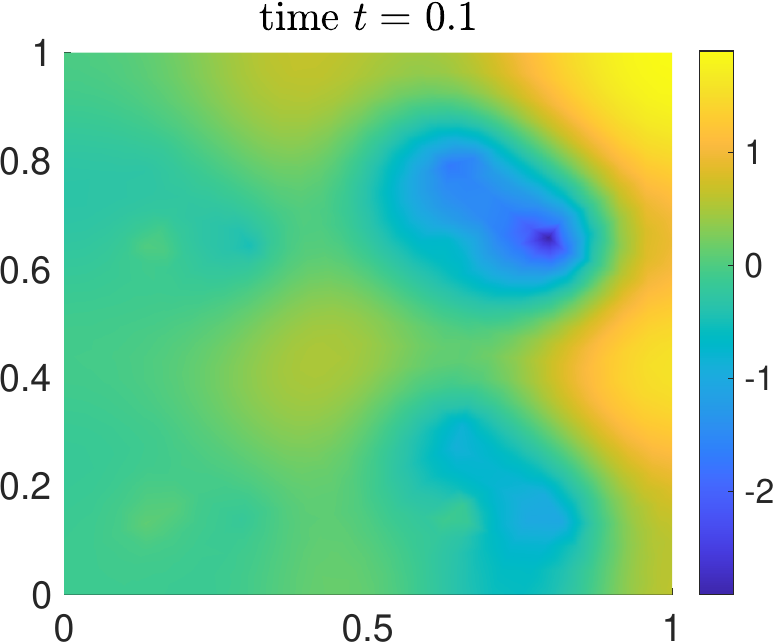}}
\subfigure[$\rho=1$]
{\includegraphics[width=0.32\textwidth,height=0.28\textwidth]{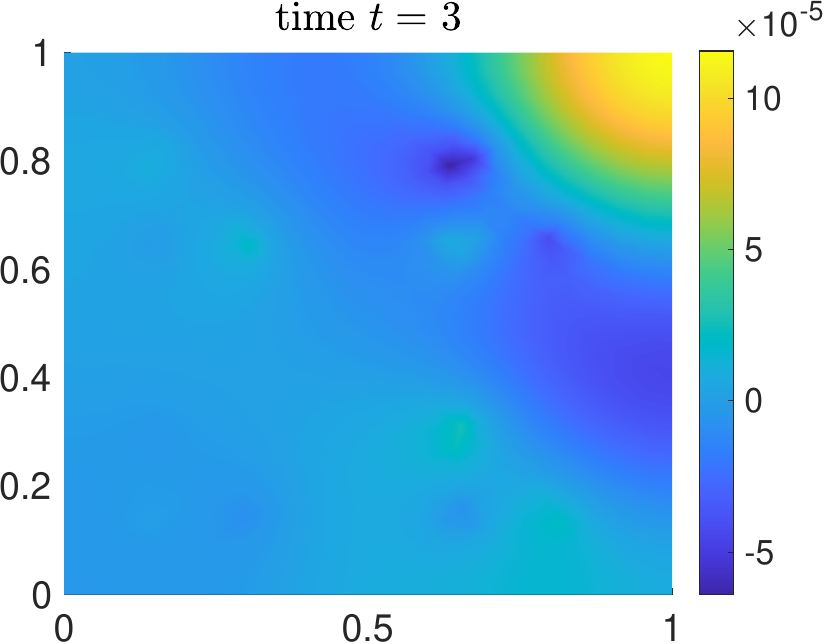}}
\subfigure[$\rho=1$]
{\includegraphics[width=0.32\textwidth,height=0.28\textwidth]{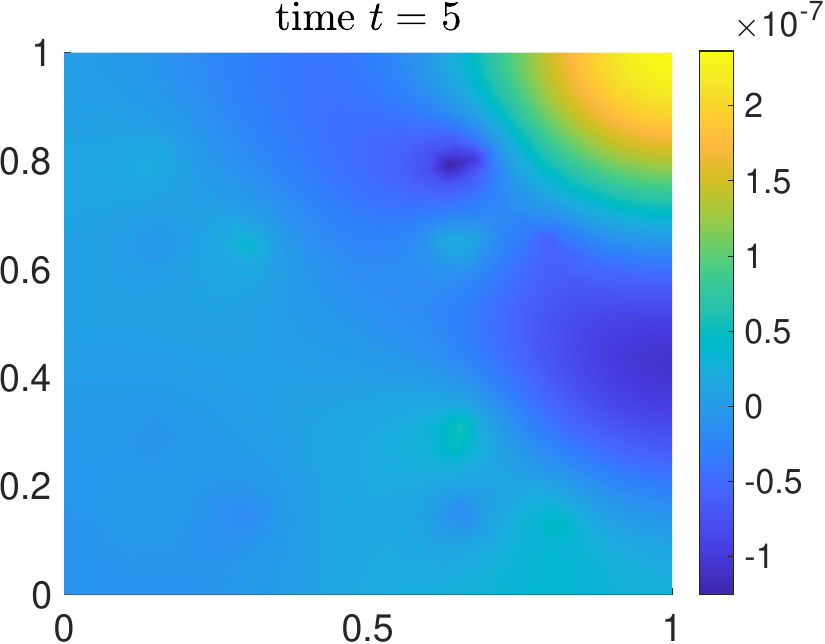}}
\\
\subfigure[$\rho=2$]
{\includegraphics[width=0.32\textwidth,height=0.28\textwidth]{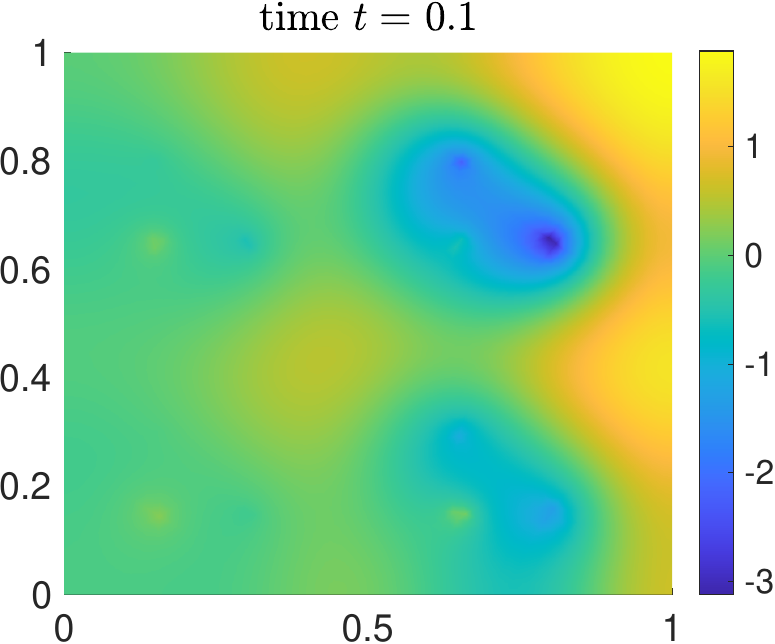}}
\subfigure[$\rho=2$]
{\includegraphics[width=0.32\textwidth,height=0.28\textwidth]{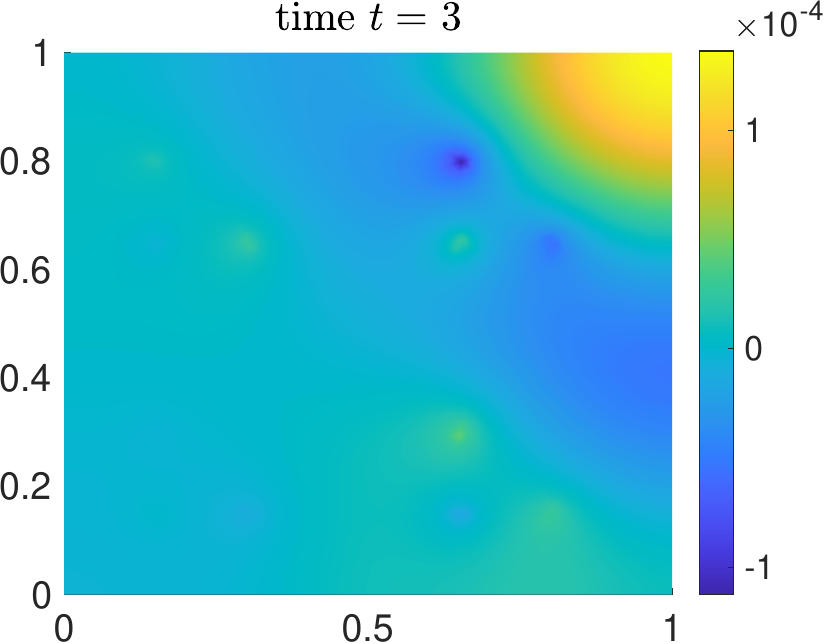}}
\subfigure[$\rho=2$]
{\includegraphics[width=0.32\textwidth,height=0.28\textwidth]{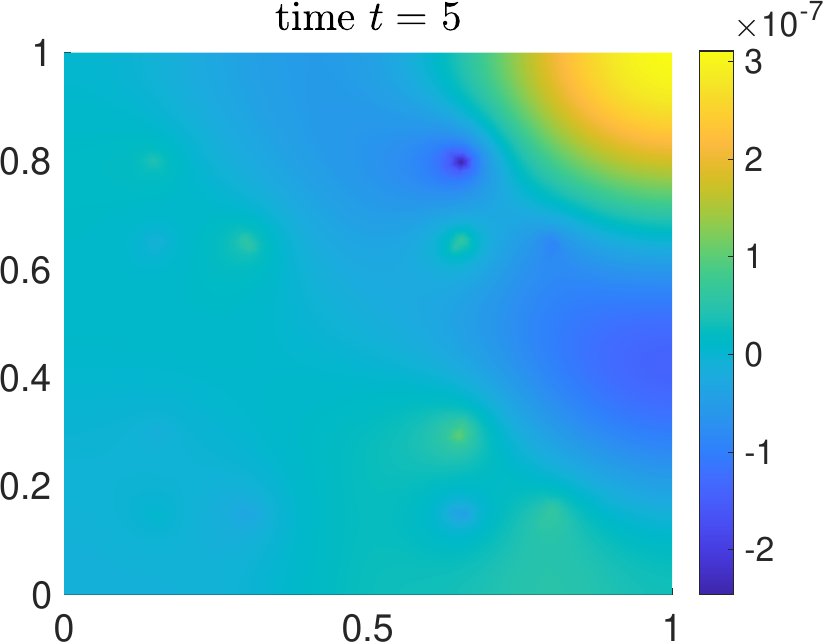}}
\\
\subfigure[$\rho=3$]
{\includegraphics[width=0.32\textwidth,height=0.28\textwidth]{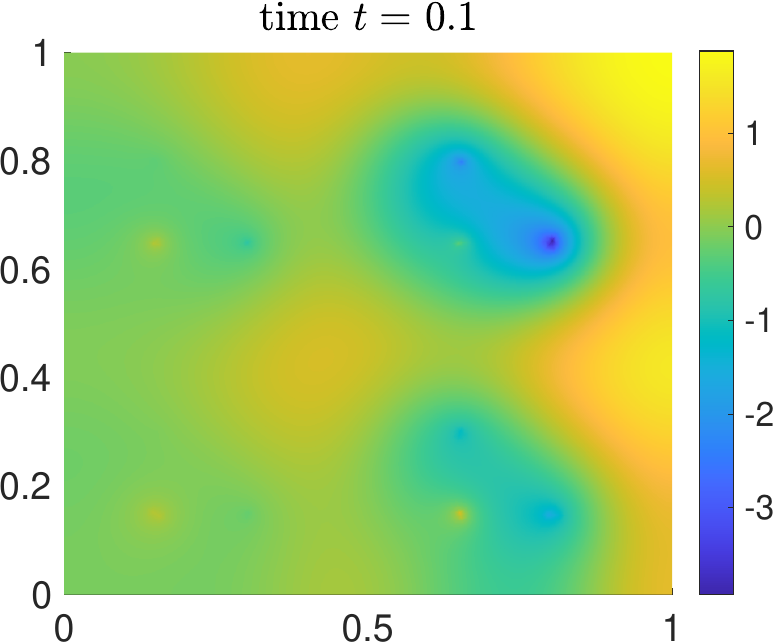}}
\subfigure[$\rho=3$]
{\includegraphics[width=0.32\textwidth,height=0.28\textwidth]{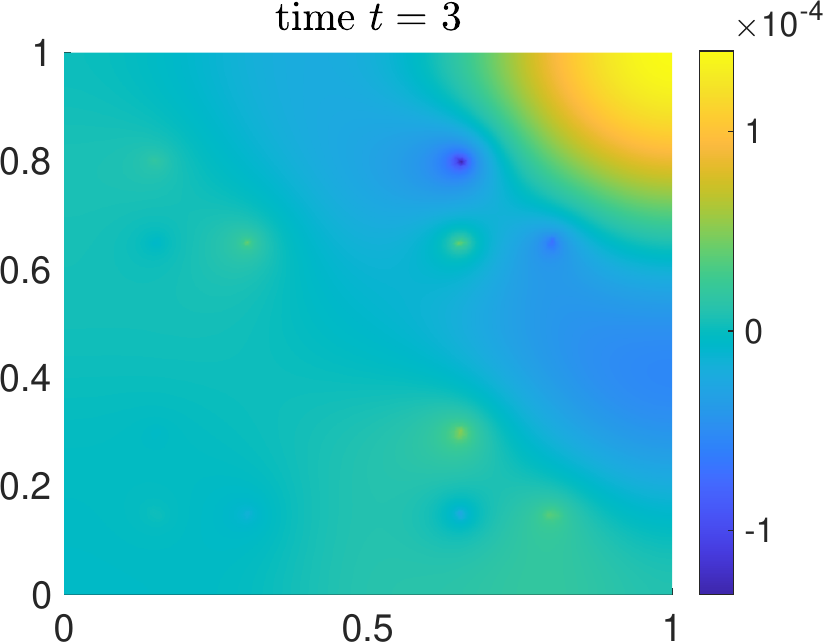}}
\subfigure[$\rho=3$]
{\includegraphics[width=0.32\textwidth,height=0.28\textwidth]{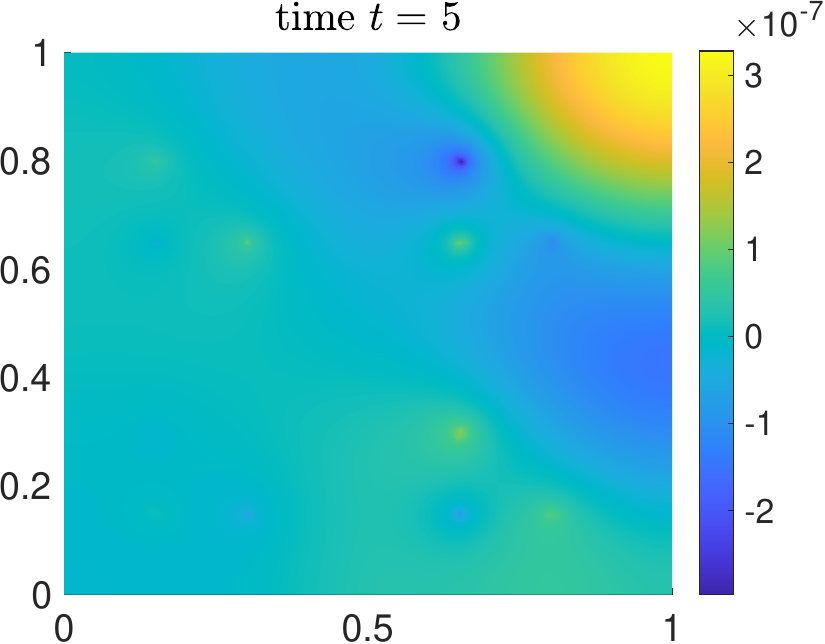}}
\caption{Time snapshots for several spatial refinements; Explicit feedback; $M=2$. Data~\eqref{param.numExpl}.}
\label{fig:timesnap-rho}
\end{figure}

\subsection{Explicit versus Riccati feedback for the autonomous case}
We have seen that the explicit feedback corresponding to~$3$ actuators located as in Fig.~\ref{fig:mesh0M1} is not able to stabilize the system. However, such~$3$ actuators could be able to stabilize the system if we take a different feedback.
We illustrate this point, in the particular autonomous setting, comparing the explicit feedback with the more computationally expensive  classical Riccati feedbacks.

First of all, solving (backwards) the differential Riccati equation for $t\in[0,+\infty)$ is not possible. Secondly, solving it in a large time interval~$t\in[0,T]$ (for a suitable final condition at time~$t=T$, e.g., see~\cite{KroRod15,PhanRod18-mcss,BreKunRod17}) can be a  very time-expensive task.

Therefore, we restrict the comparison to the autonomous case, where the solution of the Riccati equation is independent of time and we have ``just'' to solve a single algebraic Riccati equation, namely, we shall compute the solution of
\begin{subequations}\label{Riccati-Disc-alg}
\begin{align}
&0= \mathbf\Pi\overline{L_0}+\overline{L_0}^\top\mathbf\Pi -\mathbf\Pi\fkB_\beta \fkB_\beta^\top\mathbf\Pi+\fkC^\top\fkC
\intertext{with}
&\overline{L_0}=\bfM^{-1}\bfL^{\mu_{\rm ric}}(0),\quad\fkB_\beta=\bfM^{-1}\bfB_\beta\quad\mbox{and}\quad\fkC=\mathbf\Theta_\bfc[c_{(n,j)}]^\top\bfM,
\end{align}
\end{subequations}
corresponding to taking, as reaction and convection coefficients in system~\eqref{sys-y-parab-num-expl}, the evaluation at initial time~$t=0$ of the functions in~\eqref{param.numExpl}, that is,
\begin{subequations}\label{param.numExpl-aut}
\begin{equation}
a(x)-1=-3-(2-x_1)\cos(\pi x_2)-2\norm{\sin(x_2)}{\bbR},\quad
b(x)=\begin{bmatrix}
2x_1(x_1-1)x_2\\
-(x_1-\frac12)x_2(x_2-1)
\end{bmatrix}.
\end{equation}
The diffusion coefficient and initial state are taken as in~\eqref{param.numExpl}. Finally, we have taken the~$3$ actuators located as in Fig.~\ref{fig:mesh0M1} and set~$M_1=30$ for the dimension of the observation operator~$C=P_{\clE_{M_1}^\rmf}$. That is,
\begin{equation}
M_0=M_\sigma=3\quad\mbox{and}\quad M_1=30.
\end{equation}
\end{subequations}

So, the following simulations correspond to an autonomous system, with a feedback input~$u(t)=u(y(t))\in\bbR^{M_0}$ either constructed from the solution of the algebraic Riccati equation,
$u=-\beta^{-1}B^*\Pi y$, or given explicitly as in~\eqref{sys-y-parab-num-expl}.

We have computed the solution~$\mathbf\Pi$ of the Riccati equation~\eqref{Riccati-Disc} only for the coarse triangulation~$\clT^0$ in Fig.~\ref{fig:mesh0M1}, and then used such solution to construct an appropriate feedback to perform the similations in refinements of~$\clT^0$.

Let us write~$\mathbf\Pi^0=\mathbf\Pi$ to underline that~$\mathbf\Pi^0$ corresponds to the solution computed the coarse triangulation~$\clT^0$.  Following the notations in Section~\ref{S:discRicc}, we see that for the simulations in such triangulation the feedback control is given by
\begin{equation}\label{ric:co-2}
\bfB^0\bfu^0(t)\quad\mbox{with}\quad\bfu^0(t)=-\beta^{-1}\fkB^0\mathbf\Pi^0\overline y^0(t)
\end{equation}
where
$
\bfB^0=[b_{(n,j)}^0]\in\bbR^{N_0\times M_0}
$
is the matrix containing, in its columns, the vectors~$b_{(n,j)}^0$ in the coarse mesh corresponding to the actuators, $\bfM_0$ is the mass matrix in the coarse mesh, ~$\fkB^0\coloneqq(\bfM^{-1}_0\bfB^0)^\top$, and~$\overline y^0(t)\in\bbR^{N_0\times 1}$  is the vector state at time $t\ge0$.

The matrix~$\bfM^{-1}_0\bfB^0$ (and its transpose) as well as the matrix~$\overline{L_0}$ in~\eqref{Riccati-Disc-alg}  have been computed prior to the time-stepping to avoid
inverting the mass at each time step.
The computation of~$\overline{L_0}$ in fine meshes can be a demanding numerical task, as well as the computation of the solution of the algebraic Riccati equation~\eqref{Riccati-Disc-alg}.
That is why we did such computations in the coarse triangulation~$\clT^0$ only.
Next, we show how we have used the Riccati solution~$\mathbf\Pi^0$, computed in~$\clT^0$, for simulations in refinements of~$\clT^0$.

Observe that~$\bfu^0(t)$  gives us the feedback coordinates used to tune the actuators. For simulations in a refined mesh such coordinates were computed as follows.
Let $\bfB^\rho=[b_{(n,j)}^\rho]\in\bbR^{N_\rho\times M_0}$, $\rho\in\{1,2,3\}$
be the matrix containing the actuators vectors~$b_{(n,j)}^\rho$ in the mesh~$\clT^\rho$, with~$N_\rho>N_0$  points, which is a refinement~$\clT^0$. Let~$\overline y^\rho(t)\in\bbR^{N_\rho\times 1}$  be the vector state at time $t\ge0$ . We assume that the first~$N_0$ indexes correspond to points in the coarse mesh, that is,
the coordinates of~$\clM^0\overline y^\rho\in\bbR^{N_0\times 1}$, defined as~$(\clM^0\overline y^\rho)(j,1)\coloneqq\overline y^\rho(j,1)$, $1\le j\le N_0$,  correspond to the values of~$\overline y^\rho$ at the points in the coarse mesh. Then, we compute the feedback control on the finer mesh as
\begin{equation}\label{ric:co-2-fi}
\bfB^\rho\bfu^\rho(t)\quad\mbox{with}\quad\bfu^\rho(t)=-\beta^{-1}\fkB^0\mathbf\Pi^0\overline \clM_0\overline y^\rho.
\end{equation}

In Fig.~\ref{fig:Expl_vs_Ric.Expl} we see that the explicit feedback with the~$3$ delta actuators localized as in Fig.~\ref{fig:mesh0M1}, and with the control parameter gain~$\lambda=10$, is not able to stabilize the resulting autonomous system.
\begin{figure}[ht]
\centering
\subfigure%[]
{\includegraphics[width=0.45\textwidth,height=0.31\textwidth]{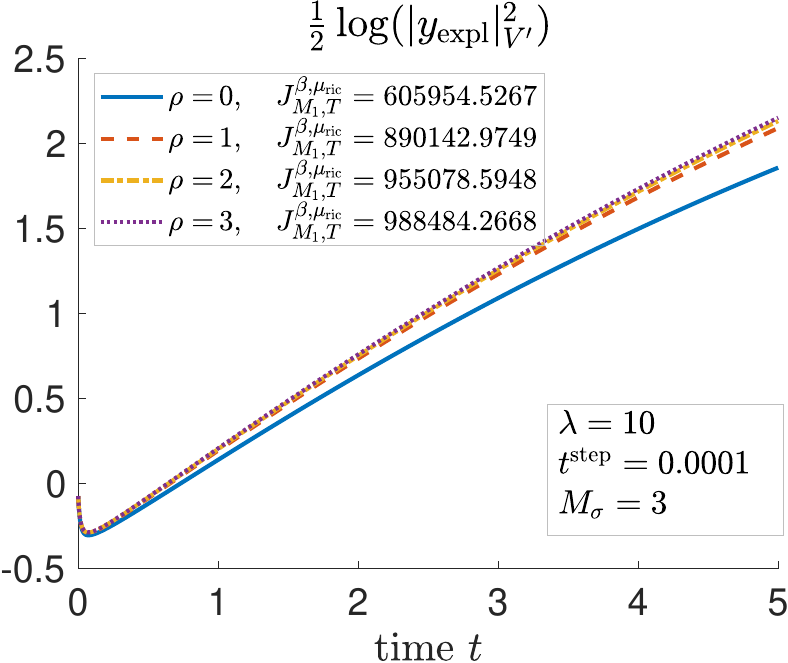}}
\quad
\subfigure%[]
{\includegraphics[width=0.45\textwidth,height=0.31\textwidth]{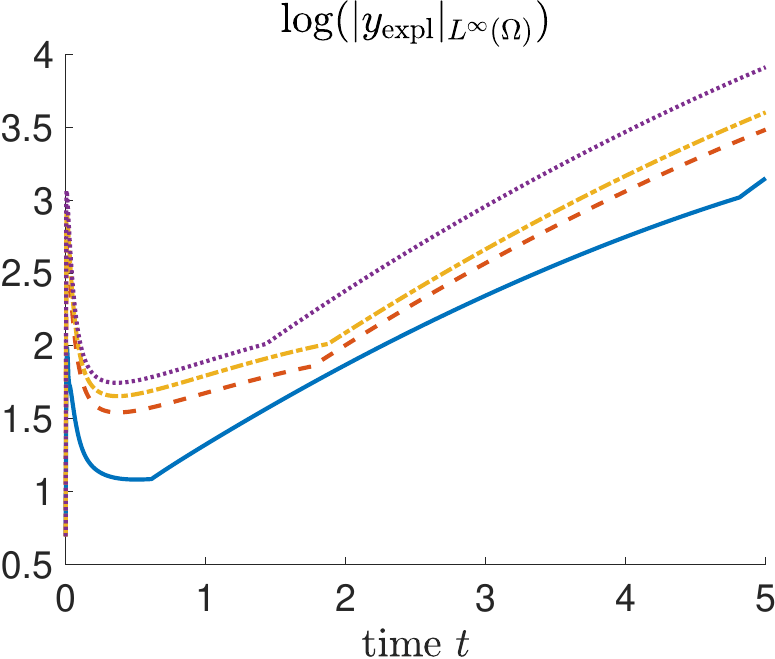}}
\caption{With explicit feedback; $M=1$. Data~\eqref{param.numExpl-aut}.}
\label{fig:Expl_vs_Ric.Expl}
\end{figure}

 The same actuators allow us to stabilize the system if we take the classical Riccati feedback control as illustrated in Fig.~\ref{fig:Expl_vs_Ric.Ricc},
\begin{figure}[ht]
\centering
\subfigure%[]
{\includegraphics[width=0.45\textwidth,height=0.31\textwidth]{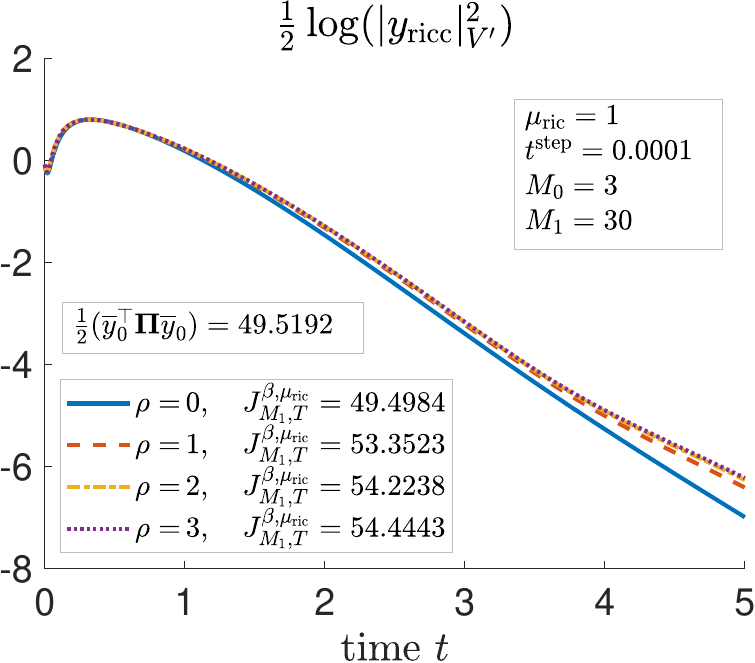}}
\quad
\subfigure%[]
{\includegraphics[width=0.45\textwidth,height=0.31\textwidth]{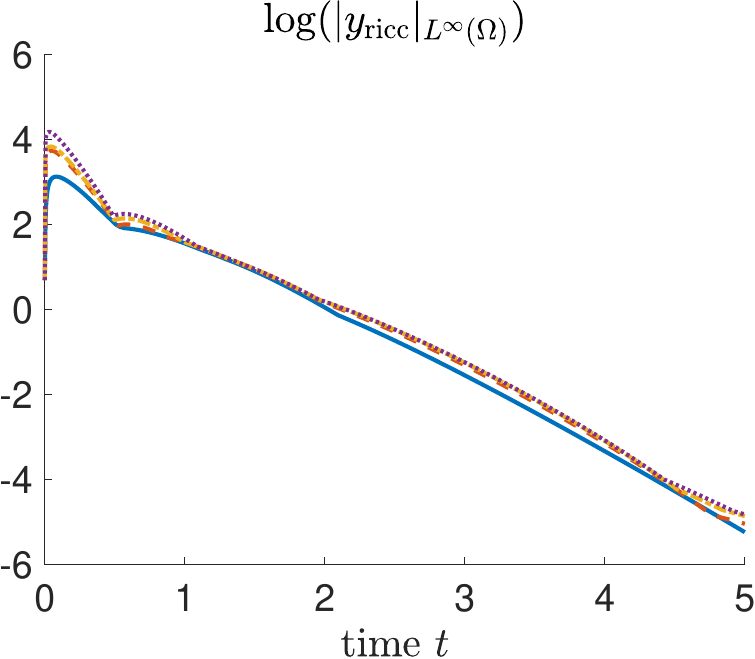}}
\caption{With Riccati feedback computed for~$\rho=0$; $M=1$. Data~\eqref{param.numExpl-aut}.}
\label{fig:Expl_vs_Ric.Ricc}
\end{figure}
 where we can also see that the feedback computed as in~\eqref{ric:co-2-fi} by using the Riccati solution in the coarse triangulation~$\clT^0$ in Fig.~\ref{fig:mesh0M1}, gives us a stabilizing control in consecutive refinements~$\clT^\rho$ of~$\clT^0$. Furthermore, we observe convergence of the norm evolution to a stable behavior as we increase the number of refinements. Hence, validating the construction we propose in~\eqref{ric:co-2-fi}. Of course, we cannot expect that such strategy/construction will work for an arbitrary initial coarse mesh, hence the simulations also show  that the coarse triangulation~$\clT^0$, in Fig.~\ref{fig:mesh0M1}, is fine enough for such strategy.

Observe also that the slope of lines in Fig.~\ref{fig:Expl_vs_Ric.Ricc} are close to (likely, a bit smaller than)~$-1=-\mu_{\rm ric}$ (for large time), where~$\mu_{\rm ric}$ is the value we have chosen in~\eqref{Riccati-Disc-alg}, hence to ``guarantee'' a stability rate (at least)~$\mu_{\rm ric}$. Note also that the evolution of the norm of the solution is not strictly decreasing. This is common for solutions~$y_{\rm ricc}$ under the action of a Riccati based feedback, where the stability holds as
\begin{equation}\notag
\norm{y_{\rm ricc}(t)}{V'}\le C\ex^{-\mu (t-s)}\norm{y_{\rm ricc}(s)}{V'},
\end{equation}
with a ``transient bound'' constant~$C$ (often) strictly larger than~$1$.

In Figs.~\ref{fig:Expl_vs_Ric.Expl} and~\ref{fig:Expl_vs_Ric.Ricc} we also see the value of the truncated cost
\begin{equation}\notag
J_{M_1,T}^{\beta,\mu_{\rm ric}}\coloneqq\tfrac12\norm{\rme^{\mu_{\rm ric}\Bigcdot}P_{\clE_{M_1}^\rmf}y}{L^2((0,T),L^2(\Omega))}^2+\beta\tfrac12\norm{\rme^{\mu_{\rm ric}\Bigcdot}u}{L^2((0,T),\bbR^{M_0})}^2,
\end{equation}
where~$T=5$ is the time we run the simulations up to; see~\eqref{Jmu-ori}.
Furthermore, we also show the approximated value for the optimal cost~$\frac12\langle\Pi y_0,y_0\rangle_{V,V'}$ given by~$\frac12(\overline y_0^\top\mathbf\Pi\overline  y_0)$.

Next, for the sake of completeness we computed the solution of the Riccati equation for the mesh after one refinement, $\rho=1$, and  used it for simulations on the refinements corresponding to~$\rho\in\{2,3\}$. The results are shown in Fig.~\ref{fig:Expl_vs_Ric1.Ricc}. We see that we obtain the same qualitative behavior as in Fig.~\ref{fig:Expl_vs_Ric.Ricc} where we used the Riccati solution computed in the coarser mesh corresponding to~$\rho=0$. Thus we conclude again that the coarser mesh is likely fine enough to capture the behavior of the norm of the state solving the optimal control problem. It is also interesting to observe that the corresponding truncated costs~$J_{M_1,T}^{\beta,\mu_{\rm ric}}$ in Figs.~\ref{fig:Expl_vs_Ric1.Ricc} and~\ref{fig:Expl_vs_Ric.Ricc} are close to each other.
\begin{figure}[ht]
\centering
\subfigure%[]
{\includegraphics[width=0.45\textwidth,height=0.31\textwidth]{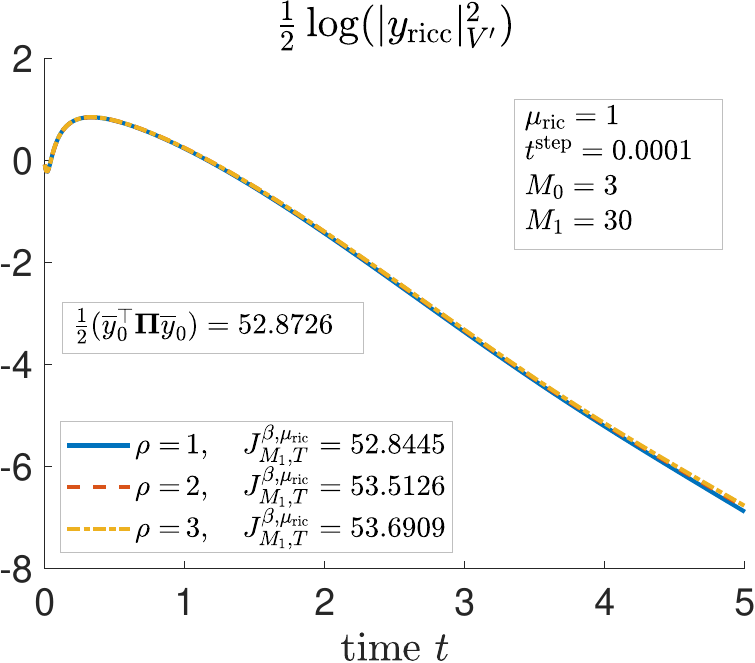}}
\quad
\subfigure%[]
{\includegraphics[width=0.45\textwidth,height=0.31\textwidth]{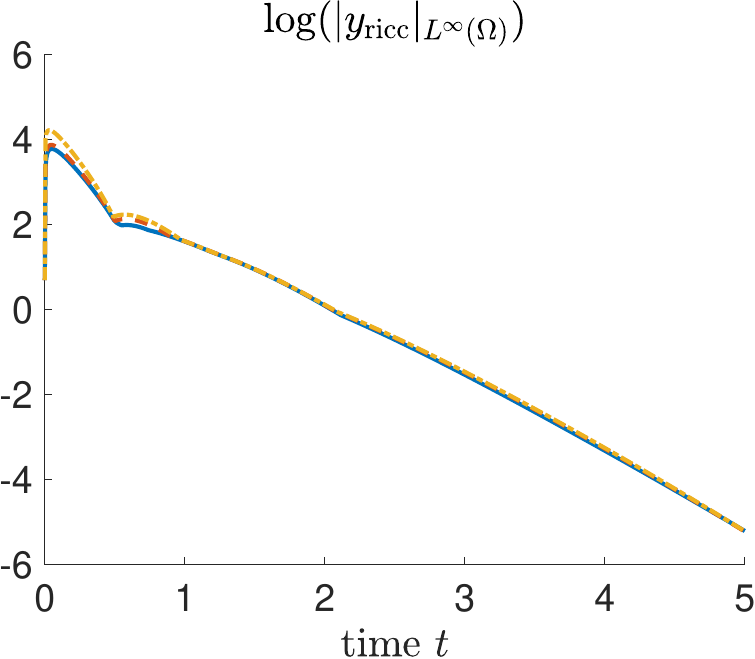}}
\caption{With Riccati feedback computed for~$\rho=1$; $M=1$. Data~\eqref{param.numExpl-aut}.}
\label{fig:Expl_vs_Ric1.Ricc}
\end{figure}

%%%%%%%%%%%%%%%%%%%%%%%%%
%%%%%%%%%%%%%%%%%%%%%%%%%
\subsection{Further remarks}
We discuss points related to the Riccati and explicit feedback.

 %%%%%%%%%%%%%%%%%%%%%%%%
 \subsubsection{On the algebraic matrix equations}\label{sS:uniqe-mat-ARE}
Since~$\fkC$ is positive semi-definite then~$\mathbf\Pi$ is the unique exponentially stabilizing symmetric positive semidefinite solution of~\eqref{Riccati-Disc-alg}. Indeed, if~$\mathbf\Pi_1$ is another stabilizing symmetric positive semidefinite solution, with~$\bfD\coloneqq\mathbf\Pi-\mathbf\Pi_1$, we find
\begin{align}
0&=\bfD\overline{L_0}+\overline{L_0}^\top \bfD -\mathbf\Pi\fkB_\beta \fkB_\beta^\top\mathbf\Pi+\mathbf\Pi_1\fkB_\beta \fkB_\beta^\top\mathbf\Pi_1\notag\\
&=\bfD\overline{L_0}+\overline{L_0}^\top \bfD -\bfD\fkB_\beta \fkB_\beta^\top\mathbf\Pi-\mathbf\Pi_1\fkB_\beta \fkB_\beta^\top\bfD\notag\\
&=\bfD X+X_1^\top \bfD,\notag
\quad\mbox{with}\quad
X\coloneqq\overline{L_0}-\fkB_\beta \fkB_\beta^\top\mathbf\Pi\quad\mbox{and}\quad X_1\coloneqq\overline{L_0}-\fkB_\beta \fkB_\beta^\top\mathbf\Pi_1.\notag
\end{align}
From the stability of~$X$ and~$X_1$ we have that all the eigenvalues of~$X$ and of~$X_1$ have negative real part
which implies
 that~$\bfD=0$; see~\cite[Thm.~1]{Chu87} \cite[Thm.~1]{Penzl98}.

 %%%%%%%%%%%%%%%%%%%%%%%%
 \subsubsection{On the choice of~$M_1$}\label{sS:rank-fact}
 The solver we used looks for low-rank (when compared to the number~$N$ of mesh points) approximations~$\mathbf\Pi_\bff$ of the Cholesky factor~$\mathbf\Pi_\bfc$ of the Riccati solution, more precisely, so that~$\mathbf\Pi\approx\mathbf\Pi_\bff^\top\mathbf\Pi_\bff$. Numerical evidence shows that (in the autonomous case ) the existence of such low-rank factors can be expected when the ranks of~$\bfB\in\bbR^{N\times M_0}$ and~$\bfC\in\bbR^{M_1\times N}$, respectively,~$M_0$ and~$M_1$, are small  (again, when compared to~$N$). We have chosen~$M_1=30$ in our simulations for which we obtained a full-rank factor~$\mathbf\Pi_\bff\in\bbR^{N\times N}$, with~$N=185$ for the case~$\rho=0$, this also gives us a positive definite approximation~$\mathbf\Pi_\bff^\top\mathbf\Pi_\bff$ for the solution~$\mathbf\Pi$. For the refined mesh we obtained a rectangular factor~$\mathbf\Pi_\bff\in\bbR^{265\times697}$. This of course does not give us a positive definite approximation~$\widetilde{\mathbf\Pi}\coloneqq\mathbf\Pi_\bff^\top\mathbf\Pi_\bff$, but it gives us (or, can still give us) a good approximation for~$\mathbf\Pi$ and for
the product control input operator~$\bfK= -\beta^{-1}(\bfM^{-1}\bfB)\widetilde{\mathbf\Pi}$ in~\eqref{ric:co-2}.
 An explanation for this phenomenon of existence of  such low-rank factor~$\mathbf\Pi_\bff$ approximations is as follows. Note that~$\Pi$ associated to the cost functional~\eqref{clJ} does not necessarily define a scalar product, because the optimal cost can vanish for suitable initial conditions. For example, for the autonomous dynamics ($A$ is independent of time)
 \begin{equation}\notag
 \dot y=-Ay+Bu,\qquad y(s)=y_s,
 \end{equation}
we can see that, if $j>M_1$ and~$y_s=e_j$, then the minimum of~$\clJ$ in~\eqref{clJ} vanishes (by taking~$u=0$). Since the dynamics is autonomous, we know that~$\Pi$ is independent of time and satisfies~$\langle\Pi e_j,e_j\rangle_{V,V'}=0$. Hence, it does not define a scalar product. In this situation,   we can expect that a numerical approximation~$\mathbf\Pi\in\bbR^{N\times N}$ of~$\Pi$ can be itself well approximated by products as~$\widetilde{\mathbf\Pi}={\widetilde{\mathbf\Pi}_\bff}^\top\widetilde{\mathbf\Pi}_\bff$ of low-rank  factors~$\widetilde{\mathbf\Pi}_\bff\in\bbR^{n\times N}$ (i.e., with $n$ small when compared~$N$); see also discussions in~\cite[Sect.~4]{BennerLiPenzl08} and~\cite{LiWhite02}.

%%%%%%%%%%%%%%%%%%%%%%%%
\subsubsection{On the computation of the Riccati feedback in fine discretizations}
Recalling~\eqref{Riccati-Disc}, we need the inverse of the mass matrix to construct~$\overline{L}=\bfM^{-1}\bfL^{\mu_{\rm ric}}$ and~$\fkB_\beta=\bfM^{-1}\bfB_\beta$. The latter is not problematic if the number of actuators~$M_0$ is small, the former is also of no concern for coarse discretizations, but it can be an issue (at least, very time expensive) when working on fine discretizations, where the number of columns of~$\bfL^{\mu_{\rm ric}}$ is large.
Circumventing this issue, we computed the finite-element Riccati solution on a coarse triangulation of the spatial domain, and then used this solution,  in the form~\eqref{Riccati-Disc}, for simulations on refinements of such a triangulation. Of course, the size (i.e., number of points) of such a coarse mesh is expected to depend on the parameters of the free dynamics (e.g., on how unstable the dynamics is). This means that for parameters requiring a large size coarse mesh we may need a different strategy.
At this point, we must say that in order to avoid the inverse of the mass matrix in the Riccati equation we could think of computing first
the product~$\mathbf\Xi=\bfM^{-1}\mathbf\Pi\bfM^{-1}$, and then recover~$\mathbf\Pi=\bfM\mathbf\Xi\bfM$. Note that from~\eqref{Riccati-Disc} we find that~$\mathbf\Xi$ solves the more general equation as follows
\begin{equation}\notag
0\approx\bfM\dot{\mathbf\Xi}\bfM+\bfM\mathbf\Xi\bfL^{\mu_ric}+(\bfL^{\mu_ric})^\top \mathbf\Xi\bfM
-\bfM\mathbf\Xi\bfB_\beta \bfB_\beta^\top \mathbf\Xi\bfM
+\bfC^\top\bfC
\end{equation}
(cf.~\cite[eq.~(15)]{MalqvistPerStill18}, \cite[below eq.~(8)]{BreitenDolgovStoll-na20}, \cite[eq.~(4.8)]{Heiland16}). Thus, we can write the equation avoiding the inverse of the mass matrix,  but we have to compute the solution of a more general equation, which is numerically more involved. We refer the interested reader to the works mentioned above and references therein.

%%%%%%%%%%%%%%%%%%%%%%%%
\subsubsection{Explicit versus Riccati}
We conclude this section with comments on a comparison between explicit and Riccati  feedbacks. Some advantages of the Riccati feedback are:
\begin{itemize}
\item for a given number of actuators,  it may succeed to stabilize the system when the explicit feedback fails;
\item it gives us the solution minimizing a classical energy functional;
\end{itemize}
while some advantages of the explicit feedback are:
\begin{itemize}
\item it is less expensive to compute than Riccati;
\item its computation does not require knowledge of the reaction-convection operator;
\item its computation cost is essentially the same for autonomous and nonautonomous systems, while for general nonautonomous systems, it is  impossible to solve the Riccati equation  in the entire time-line $t\in[0,+\infty)$.
\end{itemize}

For the particular case of nonautonomous time-periodic dynamics, the solution of the Riccati equation will be time-periodic. In this  case it is possible to solve such equation, by looking for a solution in a finite time interval with length equal to the time-period. Computing the solution is still an expensive numerical task, but if we succeed, then the resulting feedback could (or, will likely) stabilize the system when the explicit one fails.

%%%%%%%%%%%%%%%%%%%%%%%%%%
%%%%%%%%%%%%%%%%%%%%%%%%%%
%%%%%%%%%%%%%%%%%%%%%%%%%%
\section{Final remarks}\label{S:finalremks}
We have shown that stabilizability in the distribution space norm~$V'=\rmD(A^\frac12)'$ can be achieved for parabolic-like equations, with distribution actuators in~$\rmD(A)'$, where~$A$ is a diffusion-like operator in a Hilbert pivot space~$H$, with domain~$\rmD(A)\subset H$.

As an application we considered the case of scalar parabolic equations in a spatial domain~$\Omega\subset\bbR^d$, with~$d\in\{1,2,3\}$, where the actuators are delta distributions~$\deltafun_{x^i}\in\rmD(A)'$ of a given finite set of spatial points~$x^i\in\Omega\subset\bbR^d$. For this particular setting, in the case~$d=1$, we will actually have that~$\deltafun_{x^i}\in V'$, and the stabilizability can be achieved in the norm of the Hilbert pivot function space~$H=L^2(\Omega)$.
We have seen that the numerical simulations, for domains~$\Omega\subset\bbR^2$ confirm the stability in the distribution space norm~$V'$. Besides, the same simulations suggest that the stability also holds in the function space~$L^\infty(\Omega)$-norm. Of course, the simulations correspond to a set of selected parameters and, by this reason we cannot infer the validity of such stability in~$L^\infty(\Omega)$-norm, in general. It would be interesting to know whether or not the stability holds in some function space norm, in the cases~$d\in\{2,3\}$.

%%%%%%%%%%%%%%%%%%%%%
%%%%%%%%%%%%%%%%%%%%%
%%%%%%%%%%%%%%%%%%%%%
\bigskip\noindent
{\bf Aknowlegments.}
S. Rodrigues acknowledges partial support from the State of Upper Austria and
the Austrian Science
Fund (FWF): P 33432-NBL.

%%%%%%%%%%%%%%%%%%%%%%%%%%%%
%%%%%%%%%%%%%%%%%%%%%%%%%%%%
\bibliography{DeltaStab0}
\bibliographystyle{plainurl}

\end{document}